\documentclass[12pt]{amsart}

\usepackage{amssymb,amsmath, amscd, graphicx, latexsym, enumerate,comment,amsthm,amsfonts,
caption, hyperref,booktabs}
\usepackage[export]{adjustbox}
\DeclareFontFamily{U}{mathx}{}
\DeclareFontShape{U}{mathx}{m}{n}{<-> mathx10}{}
\DeclareSymbolFont{mathx}{U}{mathx}{m}{n}
\DeclareMathAccent{\widehat}{0}{mathx}{"70}
\DeclareMathAccent{\widecheck}{0}{mathx}{"71}

\oddsidemargin .25in \theoremstyle{plain}

\newtheorem{Thm}{Theorem}[subsection]
\newtheorem{Lem}[Thm]{Lemma}

\newtheorem{Cor}[Thm]{Corollary}
\newtheorem{Prop}[Thm]{Proposition}

\newtheorem{Rem}[Thm]{Remark}

{\catcode`\@=11 \gdef\mnote#1{\marginpar{\tiny
 \tolerance\@M\spaceskip2.6\p@ plus10\p@ minus.9\p@\rm#1}}}

\usepackage{mathtools}
\usepackage{xcolor,enumerate}
\usepackage{blkarray}
\usepackage{color}
\usepackage{tikz}

\usepackage{enumitem}
\usepackage{graphicx}

\usepackage{palatino}
\usepackage[hypcap=false]{caption}
\usepackage{caption, subcaption, pinlabel}
\usepackage{color}
\usepackage{soul}
\usepackage{xcolor}

\def\Z{\mathbb Z}
\def\R{\mathbb R}
\def\C{\mathbb C}
\def\P{\mathbb P}
\def\T{\mathbb T}
\def\J{\mathcal J^f}
\def\SS{\mathcal S}

\def\cp{\hbox{${\Bbb {C P}^2}$}}
\def\cpb{\hbox{$\overline{{\Bbb {C P}^2}}$}}

\def\CP#1{\mathbb{CP}^{#1}}

\def\aa{\alpha}
\def\bb{\beta}
\def\b{\frak b}

\def\d{\delta}

\def\D{\Delta}
\def\L{\Lambda}
\def\S{\Sigma}
\def\JS{\widetilde{\J}}

\def\s{\sigma}
\def\Crit{\til\D}
\def\V{\mathcal V}
\def\tV{\widetilde{\V}}
\def\ind{\operatorname{ind}}
\def\Trans{\operatorname{Trans}}

\let\del=\partial
\let\ra=\rangle
\let\la=\langle
\let\om=\omega
\let\sm=\smallsetminus
\let\til=\widetilde
\let\:=\colon

\def\Sec{\operatorname{Sec}}

\newcommand{\Hom}{\operatorname{Hom}}
\newcommand{\Mod}{\operatorname{Mod}}
\newcommand{\PMod}{\operatorname{PMod}}
\newcommand{\Diff}{\operatorname{Diff}}

\begin{document}

\title{Monodromy factorizations of lines on del Pezzo surfaces}

\author{Mohan Bhupal and Sergey Finashin}


\keywords{Lefschetz fibrations, rational elliptic surfaces, del Pezzo surfaces, monodromy factorizations.
\newline\indent Department of Mathematics,  Middle East Technical University, Ankara, Turkey
\newline\indent bhupal@metu.edu.tr, serge@metu.edu.tr}

\begin{abstract}
We give a list of monodromy factorizations in the pure mapping class group
$\PMod(\T_{d+1})$ of a torus with $d+1$ marked points that
represent lines on a del Pezzo surface $Y=\cp\#(9-d)\cpb$ of degree $d$ for $d\le3$.
 These factorizations are lifts of a certain fixed monodromy factorization
 in $\PMod(\T_{d})$ that represents $Y$.

In the case $d=1$, discussed in more detail, we give an explicit correspondence between such factorizations and the 240 roots of $E_8=K^\perp$ (orthogonal complement in $H_2(Y)$ of the canonical class).
\end{abstract}

\maketitle

\section{Introduction}
\subsection{Motivations}
 In \cite{ko}, Korkmaz and Ozbagci gave examples of monodromy factorizations for $m\le9$ disjoint sections
 of a rational elliptic Lefschetz fibration (LF) $f\colon X\to S^2$.
Since then there have been further development using their approach; see for instance \cite{HH}.
 {In the case $m=8$, there exists another nonisomorphic factorization found by Tanaka \cite{T}.}
Our first observation shows the {\it universality} of the examples in \cite{ko,T}
with respect to isomorphisms, {under the assumption of genericity, which we will always assume (see Section~\ref{conventions} for conventions).
In particular, we answer negatively Question 1.3 in \cite{HH} on the existence of "non-holomorphic" sections in rational elliptic LFs.}

 \begin{Thm}\label{universality}
Any pair of rational elliptic LFs $f_j\colon X_j\to S^2$, $j=1,2$, endowed with $1\le m\le7$, or $m=9$ disjoint sections
$S_1^j,\dots,S_m^j\subset X_j$ are related by a diffeomorphism $\Phi\colon X_1\to X_2$, such that
$f_2\circ\Phi=\phi\circ f_1$ for some diffeomorphism $\phi\colon S^2\to S^2$ and
$\Phi(S_i^1)=S_i^2$, $i=1,\dots,m$.

For $m=8$, an extra assumption that blowing down the given sections in $X_1$ and $X_2$ give diffeomorphic $4$-manifolds
(both either $S^2\times S^2$ or $\cp\#\cpb$) leads to the same conclusion.
 \end{Thm}
 
It follows that the monodromy factorizations of $f_1$ and $f_2$ are related by a conjugation identifying the fibres of $f_i$ over the basepoints and {\it Hurwitz moves}
(corresponding to elementary modifications of the chosen systems of arcs in $S^2$).
 
 If Hurwitz moves are excluded from the equivalence relation, then the classification of monodromy factorizations becomes a more subtle task, 
and it is the subject of this paper.
Namely, the task is to classify
sets of disjoint sections $S_1,\dots,S_d\subset X$ of an elliptic LF $f\colon X\to S^2$
up to fibrewise diffeomorphisms $\Phi\colon X\to X$ covering the identity map $\phi=1_{S^2}$.
A related question is classification up to simultaneous isotopy of sections.
The fibrewise diffeomorphism task is straightforward in the case of $d=1$ (see Section~\ref{sec:elliptic}) and was analyzed much more deeply
in \cite{FL} (e.g., Cor.~3.8). The isotopy classification problem is resolved in Theorem \ref{sections-are-lines}.
Our general strategy is to classify sections in terms of monodromy factorizations by induction on $d$.
In each induction step, we suppose that for $d$
disjoint sections $S_1,\dots,S_d\subset X$ a complete list of
monodromy factorizations in $\PMod(\T_d)$ is known, and we pick one of them.
Then we describe all liftings of the chosen one
to a factorization in $\PMod(\T_{d+1})$, which corresponds to
 adding one more section $S_{d+1}\subset X\sm(S_1\cup\dots\cup S_d)$.

We interpret the task at each step as describing monodromy factorizations of lines on del Pezzo surfaces $Y_d=\cp\#(9-d)\cpb$ of degree $d$.
In particular, we show that the number $n_d$ of isotopy classes of sections $S_{d+1}$ is
the number of lines in $Y_d$. The numbers $n_d$ are known to be
$240,56,27,16,10,6,3,1$, respectively, for $d=1,\dots,8$.
In this paper, we restricted ourselves to the richest cases of $d=1,2$ and $3$.

\begin{Thm}\label{Main-Th}
The lines on a del Pezzo surface $Y$ of degree $d$ for $1\le d\le 3$
are represented 
by the $n_d$ monodromy factorizations in $\PMod(\T_{d+1})$
 listed in Tables \ref{tab:d=1}, \ref{tab:d=2} and \ref{tab:d=3}.
\end{Thm}

\subsection{Homology characterization and isotopy classification theorems}
The idea of proof of the following theorem 
was presented to us by V. Schevchishin (see Section~\ref{Schevchishin} for our exposition of it).

\begin{Thm}\label{sections-are-lines}
Consider an elliptic surface $f\colon X\to\P^1$ of type $E(1)$, 
endowed with a set of disjoint smooth sections $S_1,\dots,S_d\subset X$.
  Then there exists a simultaneous isotopy transforming the sections $S_1,\dots,S_d$ into a set of disjoint lines in $X$
 (i.e., algebraic sections) through disjoint smooth sections in $X$.
\end{Thm}

In the following three theorems, we forget the algebraic structure on $X$ and view it as just an LF $f\colon X\to S^2$ of type $E(1)$ with
pairwise disjoint sections $S_1,\dots,S_d\subset X$ for $1\le d\le8$ and a torus fibre $\T$.
We denote by {$\Sec(X)$ the space of all sections $S\subset X$}, and by $\Sec(X;S_1,\dots,S_d)$
the space of sections $S\in\Sec(X)$ in the complement of $S_1\cup\dots\cup S_d$.

 \begin{Thm}\label{Bijection}
The map $S\mapsto [S]$ gives a one-to-one correspondence between the set of isotopy classes  $\pi_0(\Sec(X;S_1,\dots,S_d))$ and the set
\[\{h\in H_2(X)\,|\,h^2=-1, h\cdot \T=1, h\cdot S_i=0,i=1,\dots,d\}.\]
 \end{Thm}

In the remaining two theorems we represent $(X;S_1,\dots,S_d)$ by
some fixed monodromy factorization 
 $t_{c_1}\cdots t_{c_{12}}$ in $\PMod(\T_d)$ {into Dehn twists}, where   $c_j\subset \T\sm Q$,  for $1\leq j\leq 12$, are simple closed curves,
 and  $Q=\{q_1,\dots,q_d\}\subset\T$ is the set of
 marked points. Here $q_i=\T\cap S_i$ for $i=1,\dots,d$.
 
 \begin{Thm}\label{factorization-general}
The set of isotopy classes $\pi_0(\Sec(X;S_1,\dots,S_d))$ 
is in one-to-one correspondence with the set of
 conjugacy classes of lifts $t_{\til c_1}\cdots t_{\til c_{12}}$ of the given factorization in $\PMod(\T_d)$
 to a factorization in $\PMod(\T_{d+1})$, {where conjugacy classes are taken with respect to elements of the kernel of the forgetful homomorphism
 $\PMod(\T_{d+1})\to\PMod(\T_{d})$.}
 \end{Thm}
 
 We present explicit lists representing conjugacy classes of such lifts. To simplify notation, in what follows we abbreviate {the notation $t_c$ for Dehn twists}  to $c$.
 
\begin{Thm}\label{Main} 
Tables \ref{tab:d=1}, \ref{tab:d=2}, and \ref{tab:d=3} contain lists of $n_d$ lifts to $\PMod(\T_{d+1})$
of the following monodromy factorizations
in $\PMod(\T_d)$: 
$$(a^3b)^3 \quad\text{ for } d=1,\qquad (a_1a_2^2b_1)^3 \quad\text{ for $d=2$,} \qquad (a_1a_2a_3b_1)^3 \quad\text{ for } d=3,$$
where each isotopy class in $\pi_0(\Sec(X;S_1,\dots,S_d))$ is represented once and only once.
\end{Thm}

In Section \ref{root-section},  we present explicitly in Table \ref{tab:permutation}
the correspondence between the 240 monodromy factorizations in Table \ref{tab:d=1} and the roots of $E_8$.
One motivation for finding the correspondence with the root systems is its application indicated
in Section~\ref{concluding}, where we give more details on the action of Hurwitz moves on the factorizations.
On the one hand, we give a method to transform factorizations
in Table \ref{tab:d=1} into each other. On the other hand, we show how these moves transform our list of factorizations based on the liftings of $(a^3b)^3=1$
into lists of factorizations based on the liftings of other monodromy factorizations of $E(1)$, for instance, $(ab)^6=1$.

\subsection{Conventions and terminology}\label{conventions}
Throughout the paper, Lefschetz fibrations $f\colon X\to S^2$ are assumed to be relatively minimal, 
with only one-nodal singular fibres (aka {\it fishtails}), {of which we will assume there is at least one,}
and sections of $f$ are always smooth
(unless, in the algebraic setting, it is indicated explicitly that sections are algebraic).
We will not make a terminological distinction between sections $s\colon S^2\to X$
and their images $S=s(S^2)\subset X$.

By {\it lines} on elliptic LFs or on del Pezzo surfaces we mean rational embedded curves of square $-1$. The notation $\T$ is used for non-singular elliptic fibres (tori).

We denote by
 $F_n$ a surface $F$ with $n$ ordered marked points and by $\PMod(F_n)$ its {\it pure mapping class group}
fixing each marked point. We denote by $r_n\colon\PMod(F_{n})\to\PMod(F_{n-1})$ the {\it forgetful homomorphism} given by forgetting the last marked point in $F_{n}$.

We refer the reader to \cite{AS} for standard generalities concerning LFs, and mention below just
some special assumptions and set the notation.

\subsection{Acknowledgements} 
We thank V. Schevchishin for helpful discussions and in particular
communicating to us a scheme of proof of Theorem \ref{sections-are-lines}. {We would also like to thank the referee for many
useful comments.}

\section{Preliminaries}\label{Prelim}

\subsection{Monodromy factorization of Lefschetz fibrations}\label{subsec:monodromy}
Given a Lefschetz fibration $f\colon X\to S^2$, denote by  $\D\subset S^2$ its {\it discriminant set} formed by
its critical values. Fix a {\it basepoint} $\b\in S^2\sm\D$ and 
a collection of smoothly embedded arcs $\{A_i\}_{i=1}^n$ connecting $\b\in S^2$ with the points of $\D$, without
intersection points other than at the basepoint $\b$.
As is known, $f$ can be described by its {\it monodromy factorization} $t_{c_1}\cdots t_{c_n}=1_F$,
where the $t_{c_i}\in \Mod(F)$ are Dehn twists in the mapping class group of $F=f^{-1}(\b)$ about the simple closed curves $c_i\subset F=f^{-1}(\b)$, 
which are the  (well-defined up to isotopy) vanishing 1-cycles
of nodal degenerations of $f^{-1}(t)$ as $t$ approaches $\D$ along $A_i$.

\begin{figure}[h!]
\includegraphics[alt={Twelve singular values are placed on a circle with the basepoint at the centre. The arcs $A_i$ are
	indicated as radial arcs. The singular values are labelled by the corresponding cycles.},height=4cm]{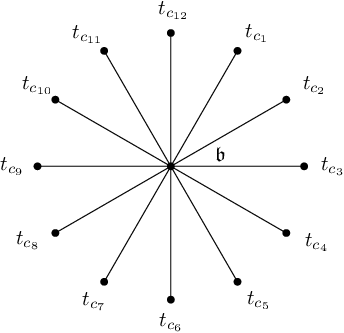}
\caption{A system of arcs in $S^2$ representing the monodromy factorization $t_{c_1}\cdots t_{c_{12}}$.}  
\label{fig:sing-scheme}
\end{figure} 
The clockwise
ordering of arcs about $\b$  gives an ordering of the curves $c_1,\dots,c_n$ (see Fig.~\ref{fig:sing-scheme}).
Note that the twists $t_{c_i}$ are applied in the reverse order in the factorization (beginning from $t_{c_n}$)
due to the functional notation for the product being used ($t_at_b=t_a\circ t_b$).

An {\it isomorphism} between two LFs $f_i\colon X_i\to S^2$, $i=1,2$, is a diffeomorphism $\Phi\colon X_1\to X_2$ commuting with the $f_i$ and some diffeomorphism $\phi\colon S^2\to S^2$.
The {\it LF-monodromy correspondence} theorem (based on \cite{EE}, and elaborated in \cite{M} and 
\cite{Mt}) provides a one-to-one correspondence
\[\text{LFs/isomorphisms}\overset{1-1}\longleftrightarrow\text{ monodromy factorizations/equivalence}\]
where equivalence of monodromy factorizations is generated by conjugations (replacing each $t_{c_i}$ by $t_{g(c_i)}=gt_{c_i}g^{-1}$ for some fixed $g\in\Mod(F)$)
and Hurwitz moves.
Recall that a Hurwitz move is a transformation of the monodromy factorization that corresponds to some elementary modification of the systems of arcs.
Namely, such a move replaces a pair of consecutive Dehn twists $t_x$ and $t_y$
in a factorization by either the pair $t_y$ and $t_y^{-1}t_xt_y$, or $t_xt_yt_x^{-1}$ and $t_x$.
Note that, in particular, {\it cyclic shifts}
replacing the factorization $t_{c_1}\cdots t_{c_n}$ of $1_F$ by 
$t_{c_k}\cdots t_{c_{k-1}}$, $2\le k\le n$, are {\it Hurwitz equivalent}, that is, obtained from each other by several Hurwitz moves.

If, for an  isomorphism $\Phi$ between LFs $f_i\colon X_i\to S^2$, $i=1,2$, we require $f_2\circ \Phi=f_1$ (or, in other words, $\phi=1_{S^2}$), we say that $\Phi$ is a {\it strong isomorphism}. 
Strong isomorphism implies, in particular, that $f_1$ and $f_2$ have the same discriminant set, and thus we can choose 
a common system of arcs and compare the associated monodromy factorizations.

A version of the LF-monodromy correspondence in this setting states:
\[\tag{*}
\label{mfc}
\begin{aligned}
&\text{$f_1$ and $f_2$ are strongly isomorphic} \Longleftrightarrow 
\text{their monodromy factorizations for}\\
&\text{the same system of arcs are conjugate by a
diffeomorphism }\  f_1^{-1}(\b)\cong f_2^{-1}(\b).
\end{aligned}
\]

Note that a {\it strong automorphism} of $X$, that is, a strong isomorphism from $X$ to itself, 
must preserve the orientation of $X$ (equivalently, of its fibres), since reversing the orientation transforms positive Dehn twists into negative ones.

\subsection{Lefschetz fibrations with sections}
Endowing an LF $f\colon X\to S^2$ 
with $d$ disjoint sections $S_1,\ldots,S_d\subset X$ gives a lifting of its monodromy representation
$\pi_1(S^2\sm\D,\b)\to\Mod(F)$ to $\PMod(F_d)$, where 
$F=f^{-1}(\b)$ is a nonsingular fibre and the lifting is taken with respect to the
forgetful homomorphism $\PMod(F_d)\to\Mod(F)$. In particular, a monodromy factorization
$t_{c_1}\cdots t_{c_n}$ of $1_F\in\Mod(F)$
is lifted to a factorization $t_{\til c_1}\cdots t_{\til c_n}$ of $1_F\in\PMod(F_d)$. Here the simple closed curve $\til c_i\subset F\sm Q$, where $Q=F\cap(S_1\cup\dots\cup S_d)$,
is isotopic to $c_i$ in $F$ for $1\leq i\leq n$.

The LF-monodromy correspondence (\ref{mfc})
holds for LFs with sections, namely,
for a pair of LFs $f_i\colon X_i\to S^2$, $i=1,2$,  endowed with disjoint sections $S_1^i,\dots,S_d^i\subset X_i$, if  from a strong isomorphism
$\Phi\colon X_1\to X_2$ we additionally require that $\Phi(S_k^1)=S^2_k$, for $k=1,\dots,d$, and from a conjugating diffeomorphism
$f_1^{-1}(\b)\cong f_2^{-1}(\b)$ we require preserving of the marked points.
Such an extension is justified by the extension of
the contractibility theorem for $\Diff_0(F)$ in \cite{EE}
to  $\Diff_0(F_m)$; see \cite{ES}.

Another consequence of the LF-monodromy correspondence describes the effect of adding one more section
$S_{d+1}\subset X\sm(S_1\cup\dots\cup S_d)$.

\begin{Thm}\label{factorization-conjugacy}
Assume that $f\colon X\to S^2$ is an LF endowed with a set of disjoint sections
$S_1,\dots, S_d\subset X$,
and $t_{c_1}\cdots t_{c_n}=1\in \PMod(F_{d})$ is
 its monodromy factorization, where the $c_i$ are simple closed curves in  $F\sm Q$. 
Then, by
associating to a section $S_{d+1}\in\Sec(X;S_1,\dots,S_d)$ a lifting $t_{\til c_1}\cdots t_{\til c_n}=1\in \PMod(F_{d+1})$
of the factorization $t_{c_1}\cdots t_{c_n}$, we obtain a well-defined one-to-one correspondence between
\begin{itemize}
\item 
the set of orbits in $\Sec(X;S_1,\dots,S_d)$ of the action
of the group of strong automorphisms $\{\Phi\colon X\to X\,|\,\Phi(S_i)=S_i, i=1,\dots,d\}$, and
\item
the set of conjugacy classes of liftings $t_{\til c_1}\cdots t_{\til c_n}$,
under simultaneous conjugation of all $t_{\til c_i}$  by elements of the kernel of $r_{d+1}:\PMod(F_{d+1})\to\PMod(F_d)$.
\end{itemize}
\end{Thm}

\begin{proof}
The LF-monodromy correspondence (\ref{mfc}) shows that the correspondence is a well-defined and injective map.
To prove its surjectivity, for any given lifting $t_{\til c_1}\cdots t_{\til c_n}=1\in \PMod(F_{d+1})$
we consider an LF $f'\colon X'\to S^2$ with disjoint sections $S_1',\dots,S_{d+1}'\subset X'$ representing it.
The correspondence (\ref{mfc}) guarantees the existence of
an isomorphism $\Phi\colon X\to X'$ such that $\Phi(S_i)=S_i'$, $i=1,\dots,d$, and after letting $S_{d+1}=\Phi^{-1}(S_{d+1}')$ we obtain a required section of $X$
whose lifting factorization is the given one.
\end{proof}

\subsection{Lefschetz thimbles and matching cycles}\label{matching}
Let $f\colon X\to S^2$ be an LF and 
consider an {\it arc} (i.e., smoothly embedded path)
$p\subset S^2$ connecting a regular value $s\in S^2$
with a critical value $\delta\in\D$.
The vanishing cycle $c\subset f^{-1}(s)$ of the nodal degeneration of $f^{-1}(t)$ as $t\in p$ approaches $\delta$
bounds  a smoothly embedded disc
$D_p\subset f^{-1}(p)$ traced by a collection of loops in $f^{-1}(p)$ collapsing to the node of $f^{-1}(\delta)$.
The disc $D_p$ is called a {\it Lefschetz thimble} along $p$.
If $X$ is endowed with
a set of disjoint smooth sections $S_1,\dots, S_d\subset X$,
then we can choose a Lefschetz thimble $D_p$ in their complement,
since sections cannot pass through the critical points.

Consider now a smoothly embedded arc $p\subset S^2$ connecting two distinct singular values $\d_0,\d_1\in\D(f)$
but otherwise disjoint from $\D(f)$.
The choice of an interior point $s\in p$
splits $p$ into a pair of arcs, $p=p_0\cup p_1$,
connecting $s$ with $\d_0$ and $\d_1$. Consider the vanishing cycles $c_i\subset  f^{-1}(s)$ bounding Lefschetz thimbles $D_{p_i}$, $i=0,1$,
in the complement of a set of fixed sections $S_1,\dots,S_d$.
The arc $p$ is called a {\it matching path} with respect to the disjoint sections $S_1,\dots,S_d$
if $c_0$ is isotopic to $c_1$ in $f^{-1}(s)$ {away from $S_1,\ldots,S_d$} (cf. \cite{S}).
In this case, we can choose the vanishing $1$-cycles so that $c_0=c_1$ to obtain a $2$-sphere $\V_p=D_{p_0}\cup D_{p_1}$,
called the {\it matching cycle} over $p$. 
The class $[\V_p]\in H_2(X\sm(S_1\cup\dots\cup S_d))\subset H_2(X)$ is called the {\it matching class} over $p$.
It is known that $\V_p^2=-2$.

Prescribing an orientation on $\del D_{p_0}=\del D_{p_1}$ will induce orientations on $D_{p_0}$ and $D_{p_1}$.
We orient $\V_p$ to match the orientation of $D_{p_0}$ and $-D_{p_1}$. So, the chosen orientation of $\V_p$ will depend on that of $\del D_{p_0}=\del D_{p_1}$
and of $p$ (which determines the direction from $\delta_0$ to $\delta_1$).

\subsection{Elliptic Lefschetz fibrations}\label{sec:elliptic}
It is well known that an elliptic LF $f\colon X\to S^2$ has
Euler characteristic $\chi(X)=12n$ for some $n\ge1$; see \cite{M}.  In particular,
 $f$ has $12n$ singular fibres, and 
its monodromy factorization has $12n$ factors $t_{c_i}$.
Such $X$ are said to be of type $E(n)$.
Moishezon proved \cite[Theorem 9]{M} that $n$ characterizes $X$ up to
isomorphism.

\begin{Thm}\label{Hurwitz-equivalence}
Let $f_i\colon X_i\to S^2$, $i=1,2$, be LFs of type $E(n)$, {and let $\b_i\in S^2$ be basepoints.}
Then any identification of the basic fibres {$f_i^{-1}(\b_i)\subset X_i$} can be extended to
an isomorphism $\Phi\colon X_1\to X_2$. In particular, 
arbitrary monodromy factorizations of the $f_i$ 
are Hurwitz-equivalent with respect to any chosen identification of basic fibres.
\qed\end{Thm}

Recall also the following well-known facts.
\begin{itemize}
\item
$E(n)$ is a fibre connected sum of $n$ copies of $E(1)$, where the latter are rational elliptic LFs diffeomorphic to $\cp\#9\cpb$.
\item
$c_1(E(n))$ is  Poincar\'e-dual to a fibre multiple $(2-n)[\T]$. 
\item
Any section of $E(n)$ has self-intersection number $-n$.
\end{itemize}

By Moishezon's Theorem \ref{Hurwitz-equivalence}, any elliptic LF $f\colon X\to S^2$ can be endowed with the structure of an elliptic surface,
which determines a group
$\Trans(X)$ formed by strong automorphisms acting on the fibres as translations with respect to the affine elliptic group structures
(see \cite[Sec.~1.3]{FL} for details).
The set of sections $\Sec(X)$ is a torsor with respect to the natural action of $\Trans(X)$, which 
implies the following.

\begin{Prop}\label{fiberwise-isotopy}
For any pair of disjoint sections $S_1,S_2$ in an elliptic LF $f\colon X\to S^2$ 
 there exists $\Phi\in\Trans(X)$
such that 
$\Phi(S_1)=S_2$.
\qed
\end{Prop}

Now we can deduce a version of Theorem \ref{Hurwitz-equivalence} for pairs $(X_i,S_i)$.

\begin{Thm}\label{H-eq-with-sec}
Under the assumptions of Theorem \ref{Hurwitz-equivalence}, consider any sections $S_i\in\Sec(X_i)$.
Then an isomorphism $\Phi\colon X_1\to X_2$ can be chosen so that $\Phi(S_1)=S_2$
and $\Phi|_{f_1^{-1}(\b_1)}$ is isotopic to any chosen
identification of the basic fibres $f_i^{-1}(\b_i)$.
 In particular, 
arbitrary monodromy factorizations of the $f_i$ in $\PMod(\T_1)$
are Hurwitz-equivalent with respect to any chosen identification of basic fibres.
\end{Thm}
\begin{proof}
We compose an isomorphism $\Phi\:X_1\to X_2$ from
Theorem \ref{Hurwitz-equivalence} with $T\in\Trans(X_2)$ sending $\Phi(S_1)$ to $S_2$, see
Proposition \ref{fiberwise-isotopy}.
\end{proof}

\subsection{Del Pezzo surfaces}
Let $S_1,\dots,S_d$ be disjoint sections of a rational elliptic LF $f\colon X\to S^2$.
As the sections are $(-1)$-spheres, we can contract them to obtain a
a {\it topological del Pezzo surface} $Y\cong \cp\#(9-d)\cpb$ of degree $d$.
The fibration $f$ descends to
 an elliptic {\it topological Lefschetz pencil} $f_Y\colon Y\dashrightarrow\CP1$ (in the sense of \cite{D})
with $d$ basepoints. It is {\it generic}, that is, its fibres (inherited from $f$) have at most one singularity, which is nodal and away from the basepoints.

Conversely, given $Y\cong \cp\#(9-d)\cpb$
endowed with a generic elliptic Lefschetz pencil
with $d$ basepoints, by blowing up the basepoints we obtain a rational elliptic LF
 supplied with a collection of $d$ disjoint sections.

\subsection{Homology classes of sections}
As before, we suppose below that $f\colon X\to S^2$ is an LF of type $E(1)$.

\begin{Prop}\label{classes-of-sections}
A homology class $h\in H_2(X)$ 
can be realized by a section if and only if $h^2=-1$ and $h\cdot \T=1$.
\end{Prop}

\begin{proof}
The necessity of $h^2=-1$ was mentioned in Section \ref{sec:elliptic}, while $h\cdot\T=1$ is evident.  
To prove realizability  by sections of classes satisfying $h^2=-1$ and $h\cdot\T=1$, 
by Moishezon's Theorem \ref{Hurwitz-equivalence},
it is enough
to consider an algebraic rational elliptic surface,
in which case realizability is well known \cite[Cor.~13]{MP}.
\end{proof}

Let us fix a section $S$ and a nonsingular fibre $\T$ in $X$, then
the orthogonal complement $\L=\la S, \T\ra ^\perp\subset H_2(X)$
is $E_8$. This is since the sublattice $\la S,\T\ra\subset H_2(X)$ spanned by
 $[S]$ and $[\T]$ is unimodular, indefinite and
contains the characteristic element $[\T]$ of $H_2(X)$, which implies that $\L$ is even, negative definite, unimodular and of rank $8$.
Using the decomposition $H_2(X)=\la S\ra\oplus\la\T\ra\oplus E_8$, and expressing for $h=x[S]+y[\T]+v$ the conditions  $h^2=-1$ and $h\cdot\T=1$,
we obtain the following.

\begin{Cor}\label{sec-classes}
A class $h\in H_2(X)$ is represented by some section in $X$ 
if and only if 
\[\pushQED{\qed} 
h=[S]+n[\T]+v,\quad n=\frac{-v^2}2\quad
 \text{for some}\ v\in E_8.
\qedhere\popQED
\]
\end{Cor}

\begin{Cor}\label{sec-roots}
A class $h\in H_2(X)$ is represented by some section in $X$ 
disjoint from $S$ 
if and only if 
\begin{equation*} 
	h=[S]+[\T]+e,\quad \text{for some}\ e\in E_8,\ \text{ with }\ e^2=-2.
\end{equation*}
In particular, there are $240$ such classes $h$ corresponding to the $240$ roots in $E_8$.
\end{Cor}

\begin{proof}
The ``only if'' part follows from
Corollary \ref{sec-classes} and $S\cdot h=0$. For the ``if'' part,
we equip $X$ with the structure of an elliptic surface and choose lines $L,L'\subset X$
realising, respectively, the classes $[S],h\in H_2(X)$. 
The translation $g\in\Trans(X)$ sending $S$ to $L$ acts as the identity on $H_*(X)$ (see, e.g., \cite[Sec.~1.5]{FL} for the Eichler transformation).
Then, $g^{-1}(L')$ is a required section of $X$.
\end{proof}

The following fact is well-known in algebraic setting. In the case of fibrewise diffeomorphisms, the proof is essentially the same.

\begin{Prop}\label{fiberwise-automorphism}
Any strong automorphism $\Phi\:X\to X$ of an elliptic LF $X$ which sends some section, $S$, to itself,
acts as the identity on $H_*(X)$.
\end{Prop}
\begin{proof}
Since $\Phi$ preserves the fibre class $\T$ and a section $S$, it is enough to justify that $\Phi_*$ acts as the identity on $E_8=\la S,\T\ra^\perp\subset H_2(X)$.
The latter holds because $E_8$ is spanned by the matching cycles $\V_p\subset f^{-1}(p)\sm S$
for some matching paths $p\subset S^2$ (see the definitions in Section \ref{matching}).
 Such spanning is a well-known fact illustrated, in particular, in Figure \ref{fig:root-graphs}, which shows the matching paths representing the basic roots of $E(8)$.
It is also trivial to see that $f^{-1}(p)\sm S$ is homotopy equivalent to a wedge of a 2-sphere $\V_p$ and a circle, and so, $[\V_p]$ generates
$H_2(f^{-1}(p)\sm S)=\Z$.
Now it is left to notice that $\Phi_*([\V_p])=[\V_p]$, since $\Phi$ preserves the orientation of the fibres of $X$.
\end{proof}

\subsection{Intersection of sections with matching cycles}
Assume that $p$ is a matching path relative to disjoint sections $S_1,\dots,S_d\subset X$ and $\V_p$ is the matching cycle.
After adding an additional section $S_{d+1}\subset X\sm(S_1\cup\dots\cup S_d)$,
we can construct Lefschetz thimbles $D_0$ and $D_1$ in the complement of $S_1\cup\dots\cup S_{d+1}$. (Here we have abbreviated $D_i=D_{p_i}$.)
Our assumption that $p$ is matching means that the
vanishing cycles $c_k=\del D_k\subset \T_{d+1}$, $k=0,1$, are isotopic in $\T_d$ (obtained from $\T_{d+1}$ by forgetting $q_{d+1}=\T\cap S_{d+1}$),
although $c_0$ and $c_1$ are possibly non-isotopic in $\T_{d+1}$.
The trace of an isotopy in $\T_d$ 
gives a $2$-chain $A$, which is a sub-surface in $\T\sm(S_1\cup\dots\cup S_d)$
(possibly with a few cross-point singularities) bounded by the $1$-cycle $\del D_1-\del D_0$.
 The union $\tV_p=D_0+A-D_1$ is a $2$-chain homologous to $\V_p$ in $X\sm(S_1\cup\dots\cup S_d)$.

\begin{figure}[h!]
\caption{Intersection index of a section $S_2$ with a matching cycle $\V_p$. Region $A\subset\T\sm\{q_1\}$ with $\del A=\del D_1-\del D_0$ is
shaded.}
\label{fig:index}
\includegraphics[alt={Three tori are drawn with $q_1$ at the top and $q_2$ at the bottom in each. In the first torus,
the lower half is shaded with the left boundary labelled $\partial D_0$ and the right $\partial D_1$.
The second torus is drawn similarly with the labels of the two boundaries switched. In the third torus,
a region is shaded that does not include $q_1$ or $q_2$.
},height=3.8cm]{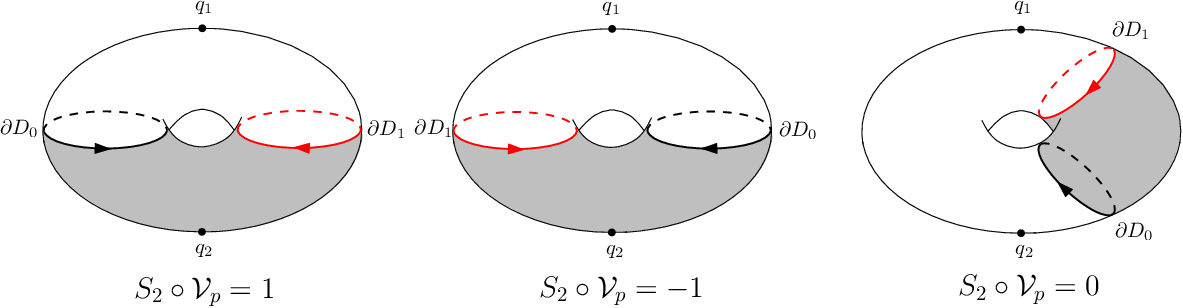}
\end{figure}

Recall that the index $\ind(q,\gamma)$ of a point $q$ with respect to a $1$-boundary
in a punctured surface is the same as the linking number of $q$ and $\gamma$, and
is defined as the intersection index $q\circ R$ of $q$ with a $2$-chain $R$, where $\del R=\gamma$.

Applying it to  $\T\sm\{q_1,\dots,q_d\}$, where
$q_i=\T\cap S_i$, $i=1,\dots,d$, and $\T$ is some nonsingular fibre,
we immediately conclude

\begin{Lem}\label{int-ind}
$S_{d+1}\circ\V_p=\ind(q_{d+1},\del D_1-\del D_0)$.
\qed\end{Lem}

An illustration in the case of $d=1$ is given in Figure~\ref{fig:index}.

\section{Monodromy factorizations of $E(1)$ with up to four sections}

\subsection{Sample factorizations and their variations}
This Section contains 
Tables~\ref{tab:d=1}--\ref{tab:d=3} of monodromy factorizations announced in Theorem \ref{Main}.
Each Table contains a list of {\it sample factorizations}.
Each factorization is a lifting of $(a^3b)^3=1\in \Mod(\T)$ to
$\PMod(\T_{d})$ and contains three {\it blocks} which are liftings of the corresponding product $a^3b$.
{\it Variations} of sample factorizations are other factorizations that can be obtained by either
\begin{itemize}\item
{\it shift-variations} cyclically permuting the three blocks, or
\item
{\it block-permutations} of letters covering the $a$'s within the same block.
\end{itemize}
The total number of variations of each sample factorization (including the sample itself) is indicated in the last column of each table, marked with ``$\#$''.
For example, the sample factorization in Row 2 of Table \ref{tab:d=1}, that is, $a_1^2a_2b_1a_1a_2^2b_1a_2^3b_2$,
admits two shift-variations: $a_1a_2^2b_1a_2^3b_2a_1^2a_2b_1$ and
$a_2^3b_2a_1^2a_2b_1a_1a_2^2b_1$. By block-permutations of the block $a_1^2a_2b_1$, we can obtain $a_1a_2a_1b_1$ and $a_2a_1^2b_1$,
while block $a_1a_2^2b_1$ gives variations  $a_2a_1a_2b_1$ and  $a_2^2a_1b_1$.
In total, all variations of the sample factorization in Row~2 provides $27$ factorizations.

\subsection{The case of a pair of disjoint sections in $X$} 

\begin{Prop}\label{factorization1+1}
Each entry in the second column of Table~\ref{tab:d=1} is a factorization of the identity in $\PMod(\T_2)$.
Furthermore, the forgetful homomorphism $r_2\colon\PMod(\T_2)\to\PMod(\T_1)$ carries each such factorization to the factorization $(a^3b)^3$.
\end{Prop}

\begin{table}[h!]
\caption{Sample factorizations of sections $S\subset X\sm S_1$.} \label{tab:tet}
\label{tab:d=1}
\hskip-2mm\scalebox{0.85}{\begin{tabular}{ccc}
\hline
Type  & Sample factorization & $\#$\\
\hline
1\phantom{*} & $a_1a_2^2b_1a_1a_2^2b_1a_1a_2^2b_1$ & $27$\\
2\phantom{*} & $a_1^2a_2b_1a_1a_2^2b_1a_2^3b_2$ & $27$\\
3\phantom{*} & $a_1a_2^2b_2a_1^2a_2b_1a_2^3b_2$ & $27$\\
4\phantom{*} & $a_1^3b_1a_2^3b_2a_2^3b_3$ & $3$\\
\hline
1* & $a_2a_1^2b_1a_2a_1^2b_1a_2a_1^2b_1$ & $27$\\
2* & $a_2^2a_1b_2a_2a_1^2b_2a_1^3b_1$ & $27$\\
3* & $a_2a_1^2b_1a_2^2a_1b_2a_1^3b_1$ & $27$\\
4* & $a_2^3b_3a_1^3b_2a_1^3b_1$ & $3$\\
\hline
5\phantom{*} & $a_2^3b_1a_3a_2^2b_2a_1a_2^2b_1$ & $27$\\
5* & $a_2^3b_2a_1a_2^2b_1a_3a_2^2b_2$ & $27$\\
\hline
6\phantom{*} & $a_2^3b_1a_2^3b_1a_3a_2a_1b_1$ & $9$\\
6* & $a_2^3b_2a_2^3b_2a_1a_2a_3b_2$ & $9$\\
\hline
\end{tabular}}
\includegraphics[alt={A standard pair of generators $a$ (meridian), $b$ (longitude), on ${\mathbb T}_1$ (once-punctured torus) is shown.
Below, liftings $a_i, b_i, i=1,2,3$ of these generators to ${\mathbb T}_2$ are shown: they are separated by punctures (marked points) $q_1,q_2$.
},width=9.5cm,valign=c]{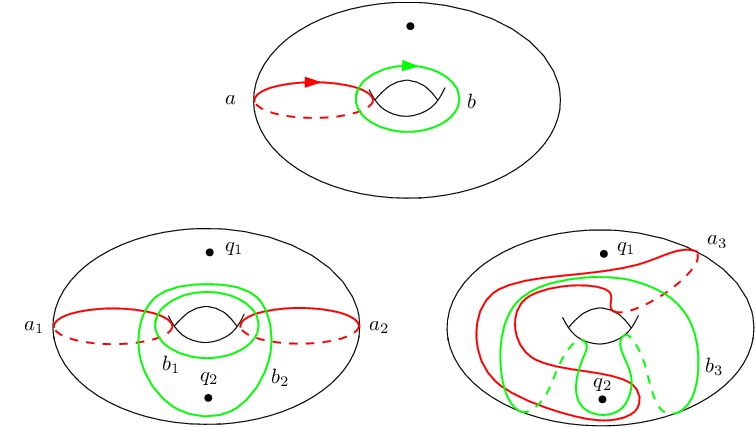}
\end{table}

\begin{proof} 
The Type 1 sample factorization corresponds to the well known identity $(a_1a_2^2b_1)^3=1\in\Mod(\T_2)$.
Using the fact that $a_1$ commutes with $a_2$ and that
\begin{equation} \label{eq:S_1^2b_1}
b_1=a_2a_1^{-1}b_{2}a_1a_2^{-1} \ 
\end{equation}
(which is obvious from Figure \ref{flat-torus}) we get from the Type 1  
the Type 2 sample factorization after substitution for the last $b_1$ and then performing a cyclic permutation.
Similarly, Type 3 is obtained from Type 2 by substitution for the first $b_1$. 
Using $b_2=a_2a_1^{-1}b_3a_1a_2^{-1}$, we obtain Type~4 from Type~3 by substituting for the first $b_2$ and then applying a shift-variation.
Type~1* corresponds to another standard relation, $(a_2a_1^2b_1)^3=1\in\PMod(\T_2)$, and Types 2*--\,4* are obtained from it in a similar way.

Using $a_i=b_2^{-1}b_1a_{i+1}b_1^{-1}b_2$ for $i=1,2$, 
we get Type 5 from Type 2 by
replacing the first three factors, using a shift-variation and a block-permutation.
Similarly, we get Type~5* from Type~3  by
replacing the factors between the first $b_2$ and the $b_1$.
Finally, we obtain Types~6 and 6* from Types~5 and 5*, respectively, by using \eqref{eq:S_1^2b_1}, and shift-variations with block-permutations as appropriate.

The second claim in the proposition is straightforward from the definition.
\end{proof}

\begin{Rem}
Theorem \ref{factorization-conjugacy} suggests a construction of different monodromy factorizations of isotopic sections $S\subset X\sm S_1$,
by applying simultaneous conjugation to all Dehn twists of a factorization by some element $g\in\ker r_2$.
Examples are replacing all $a_i$ by $a_{i+1}$ or all $b_i$ by $b_{i+1}$ in any factorization of Types 1--3, or 1*--\,3*.
In these examples the elements $g$ are represented by 
$b_1^{-1}b_2$ or $a_1a_2^{-1}$, respectively.
\end{Rem}

\subsection{Symmetries in the list of factorizations}\label{symmetries}
The monodromy factorizations in Table \ref{tab:d=1} admit a {\it reversion} 
operation induced 
by an involution on the basic torus fibre which interchanges $a_1$ with $a_2$, and $b_1$ with $b_2$.
This involution looks like 
 the central reflection on the ``flat torus'' diagram, Figure \ref{flat-torus},
where one marked point is at the centre and the other at the corners.
\begin{figure}[h!]
\caption{Basic torus fibre with marked points $q_1$  (corners) and $q_2$ (center)}
\label{flat-torus}
\includegraphics[alt={A square representing a torus is drawn. Two horizontal lines represent the curves $a_1$ and $a_2$ with $a_2$ above $q_2$ and $a_1$ below. Vertical lines show $b_1$ and $b_2$ with 
$b_1$ to the left and $b_2$ to the right.
},height=3.2cm]{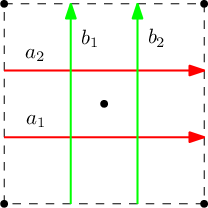}
\end{figure}
Reversion of  Types 1--6 factorizations
 are conjugate to Types 1*--\,6*.
 For instance, for the Type 1 factorization $(a_1a_2^2b_1)^3$, reversion of the indices $\rm{1}\leftrightarrow\rm{2}$ gives
$(a_2a_1^2b_2)^3$. The latter is transformed into the Type 1* factorization $(a_2a_1^2b_1)^3$ after we
apply the conjugation replacing each factor $x$ by $a_2a_1^{-1}xa_1a_2^{-1}$.

Other symmetries are represented by reflections across horizontal or vertical lines passing through one of the marked points.
The induced modification of monodromy factorizations is 
reversion of the order of factors
combined with reversion of the indices for only $a$ or only $b$
(but not for both, as was the case before).
This symmetry can be interpreted as the effect of the complex conjugation for an appropriate $J$-holomorphic structure in $X$
(or an antisymplectic involution for an appropriate symplectic structure).

\subsection{The case of three disjoint sections}

\begin{Prop}\label{factorization1+2}
Each entry in the second column of Table~\ref{tab:d=2} is a factorization of the identity in $\PMod(\T_3)$.
Furthermore, the forgetful homomorphism $r_3\colon\PMod(\T_3)\to\PMod(\T_2)$ carries each such factorization to the factorization $(a_1a_2^2b_1)^3$. 
\end{Prop}

\begin{table}[h!]
\caption{Sample factorizations of sections $S\subset X\sm (S_1\cup S_2)$.}
\label{tab:d=2}
\hskip-2mm\scalebox{0.8}{\begin{tabular}{ccc}
\hline
Type  & Sample factorization & $\#$\\
\hline
1 & {$a_1a_2a_3b_1a_1a_2a_3b_1a_1a_2a_3b_1$} & $8$\\
2 & $a_1a_2^2b_1a_1a_2a_3b_1a_1a_3^2b_2$ & $6$\\
3 & $a_1a_2a_3b_2a_1a_2^2b_1a_1a_3^2b_2$ & $6$\\
\hline
4 & $a_1a_2^2b_1a_1a_2a_3b_2\widecheck{a}_1a_2^2b_1$ & $6$\\
5 & $a_1a_2a_3b_2a_1a_2^2b_2\widecheck{a}_1a_2^2b_1$ & $6$\\
6 & $a_1a_2^2b_2a_1a_2^2b_2\widecheck{a}_1a_2a_3b_2$ & $6$\\
\hline
7 & $\widecheck{a}_1\widecheck{a}_2a_2b_1a_1a_2^2b_2\widecheck{a}_1a_2^2b_2$ & $6$\\
8 & $\widecheck{a}_1a_2^2b_1a_1a_2^2b_2\widecheck{a}_1\widecheck{a}_2a_2b_1$ & $6$\\
9 & $\widecheck{a}_1a_2^2b_1a_1a_2\widecheck{a}_2b_1\widecheck{a}_1a_2^2b_1$ & $6$\\
\hline
\end{tabular}}
\includegraphics[alt={Liftings $a_i, i=1,2,3$, $b_j$, and $\widecheck a_j, j=1,2$ are shown on ${\mathbb T}_3$. They are separated by punctures $q_i, i=1,2,3$ and are sent to the curves $a_i,b_i$ in the figure accompanying Table 1 by the forgetful map dropping $q_3$.},width=9.5cm,valign=c]{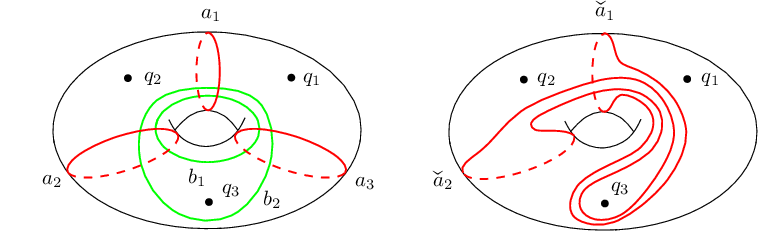}
\end{table}

\begin{proof}
The Type 1 sample factorization corresponds to the so-called {\it star relation},    
$(a_1a_2a_3b_1)^3=1\in\PMod(\T_3)$; see Gervais \cite{g}. Using 
\begin{equation} \label{eq:S_1^3b_1}
b_1=a_3a_2^{-1}b_2a_2a_3^{-1}, 
\end{equation}
we obtain Type 2 factorization from Type 1 by replacing the last $b_1$ and using a cyclic permutation. Similarly, Type 3 is obtained from Type 2 by replacing the first $b_1$.

Using 
\begin{equation} \label{eq:S_1^3a_i}
a_i=b_1^{-1}b_2\widecheck{a}_ib_2^{-1}b_1, \quad \text{for $i=1,2$, and} \quad a_3=b_1^{-1}b_2a_2b_2^{-1}b_1,
\end{equation}
we obtain Type 4 from Type $2$ by substituting for the terms between the second $b_1$ and the $b_2$.  We obtain Type 5  from Type 4  by substituting for the first $b_1$, using 
\eqref{eq:S_1^3b_1}, and Type 6 from Type 5 by substituting for the final $b_1$, and using a cyclic permutation.

We obtain Type 7 from Type 5 by substituting for the first three terms using \eqref{eq:S_1^3a_i} and
using a cyclic permutation. Finally, Types 8 and 9 are obtained from Type 7  by successively substituting
for the $b_2$'s using
$b_2=a_2^{-1}\widecheck{a}_2b_1\widecheck{a}_2^{-1}a_2.$

The second statement is straightforward.
\end{proof}

\subsection{The case of four disjoint sections}

\begin{Prop}\label{factorization1+3}
Each entry in the second column of Table~\ref{tab:d=3} is a factorization of the identity in $\PMod(\T_4)$.
Furthermore, the forgetful homomorphism $r_4\colon\PMod(\T_4)\to\PMod(\T_3)$ carries each such factorization to the factorization  $(a_1a_2a_3b_1)^3$.
\end{Prop}

\begin{table}[h!]
\caption{Sample factorizations of sections $S\subset X\sm(S_1\cup S_2\cup S_3)$.} 
\label{tab:d=3}
\hskip-2mm\scalebox{0.84}{
\begin{tabular}{ccc}
\hline
Type  & Sample factorization & $\#$\\
\hline
1 & $a_1a_2\widehat a_3 b_1a_1a_2\widehat a_3 b_1a_1\widehat a_2a_3b_1$ & $3$\\
2 & $a_1a_2\widehat a_3 b_1a_1\widehat a_2\widehat a_3 b_2a_1a_2a_3b_1$ & $3$\\
3 & $a_1\widehat a_2\widehat a_3 b_2a_1a_2\widehat a_3 b_2a_1a_2a_3b_1$ & $3$\\
\hline
4 & $a_1a_2\widehat a_3 b_2\widecheck a_1a_2a_3b_1a_1a_2a_3b_1$ & $3$\\
5 & $a_1a_2a_3b_2\widecheck a_1a_2a_3b_1a_1a_2\widehat a_3 b_2$ & $3$\\
6 & $a_1a_2a_3b_2\widecheck a_1a_2\widehat a_3 b_2a_1a_2a_3b_2$ & $3$\\
\hline
7 & $a_1a_2a_3b_2\widecheck a_1a_2a_3b_2\widecheck a_1\widecheck a_2a_3b_1$ & $3$\\
8 & $a_1a_2a_3b_2\widecheck a_1\widecheck a_2a_3b_1\widecheck a_1a_2a_3b_1$ & $3$\\
9 & $a_1\widecheck a_2a_3b_1\widecheck a_1a_2a_3b_1\widecheck a_1a_2a_3b_1$ & $3$\\
\hline
\end{tabular}}
\includegraphics[alt={Liftings $a_i, b_i, \widehat a_i$, and $\widecheck a_i$ of the curves in the figure accompanying Table 2 are shown.
Namely, the forgetful map dropping the puncture $q_4$ sends these curves to the ones in the figure accompanying Table 2.
These curves are the ones used in the 9 factorizations in Table 3.
},width=9.5cm,valign=c]{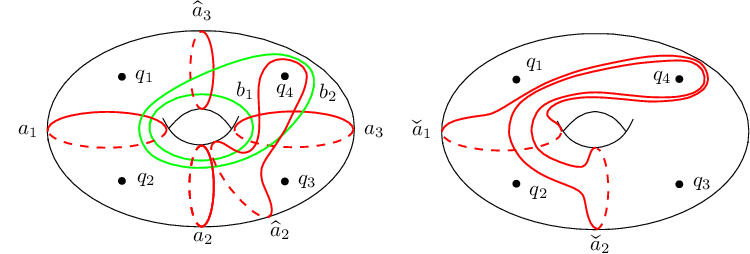}
\end{table}
\renewcommand{\baselinestretch}{1}

\begin{proof}
The Type 1 sample factorization $a_1a_2\widehat a_3 b_1a_1a_2\widehat a_3 b_1a_1\widehat a_2a_3b_1$
in $\PMod(\T_4)$ is given by Korkmaz and Ozbagci \cite{ko}. Using $b_1=\widehat a_2a_2^{-1}b_2a_2\widehat a_2^{-1}$, we 
obtain Type 2 from Type 1 by substituting for the second $b_1$. Similarly, we obtain
Type 3 from Type 2 by substituting for the first $b_1$.

Using
\begin{equation} \label{eq:S_1^4a_1}
a_1=b_2b_1^{-1}\widecheck a_1b_1b_2^{-1},\quad \widehat a_i=b_2b_1^{-1} a_ib_1b_2^{-1}, \quad 
\text{for $i=2,3$},
\end{equation}
we obtain Type 4 from Type 2  by substituting for the terms between the first $b_1$ and
the $b_2$. Using $b_1=\widehat a_3 a_3^{-1}b_2a_3\widehat a_3 ^{-1}$, we obtain Types 5 and 6 from Type 4 
by sequentially substituting for the second $b_1$ and then the first $b_1$ and using a cyclic permutation.

Using \eqref{eq:S_1^4a_1} together with $a_2=b_2b_1^{-1}\widecheck a_2b_1b_2^{-1}$, we obtain Type 7 from 
Type 5 monodromy by substituting for the terms between $b_1$ and the second $b_2$. Finally, we obtain the factorizations of Types
8 and 9 from Type 7 by successively using $b_2=a_2^{-1}\widecheck a_2b_1\widecheck a_2^{-1}a_2$.

The second statement of the proposition is again straightforward.
\end{proof}

\section{Homology classes of sections defined by monodromy factorization}

\subsection{$\D$-sequences and chains of matching cycles in $X\sm S_1$}\label{2-cycles}
Each of the $240$ factorization sequences represented in Table \ref{tab:tet} (sample factorizations together with their variations)
splits into a
$b$-{\it sequence} $b_{j_1}b_{j_2}b_{j_3}$ (the elements in the $4$th, $8$th and $12$th positions),  and an
$a$-{\it sequence} $a_{i_1}\dots a_{i_9}$ (the elements in the remaining positions enumerated consecutively).
The differences of indices $i_k-i_{k+1}$, $k=1,\dots,8$, and $i_9-i_1$
form the {\it $\Delta a$-sequence},
while the differences $j_1-j_{2}$, $j_2-j_3$ and $j_3-j_1$ form the {\it $\D b$-sequence}; see Table 
\ref{fig:E8-graph} for examples with the sample factorizations.
In the accompanying figure, we connected 
the points of $\D\subset S^2$ that represent consecutive elements of $a$- and 
$b$-sequences by edges
($p_1,\dots,p_9$ and $q_1,q_2,q_3$, respectively).
Here we are assuming that the monodromy of $(X;S_1)$ is written relative to the system of arcs shown
in Figure \ref{fig:sing-scheme}.
As a consequence, edge $p_i$ represents the $i$th term of the $\D a$-sequence, and edge $q_j$ the $j$th term of $\D b$-sequence.

\begin{table}[h!]
\caption{$\D$-sequences of sample factorizations of sections $S\subset X\sm S_1$. Accompanying figure: Matching cycles in $X\sm S_1$.}
\label{fig:E8-graph}
\label{tab:roots}
\raisebox{-28mm}{\includegraphics[alt={The vertices of a 12-gon are labelled consecutively clockwise with the letters $(a^3b)^3$. Adjacent vertices labelled $a$ are connected by line segments to form a 9-gon, and vertices labelled $b$ are connected to form a triangle. The edges represent matching 2-cycles associated with curves $a, b$ on ${\mathbb T}_1$.},height=5.7cm]{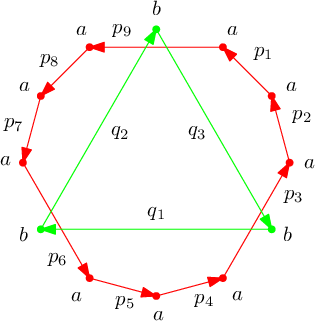}}\hskip5mm
\scalebox{0.9}{\begin{tabular}{ccccccc}
\hline
Type  & Factorization &$\D a$&$\D b$\\
\hline
1\phantom{*} & $a_1a_2^2b_1a_1a_2^2b_1a_1a_2^2b_1$ &-101-101-101&000\\
2\phantom{*} & $a_1^2a_2b_1a_1a_2^2b_1a_2^3b_2$ &0-11-100001&0-11 \\
3\phantom{*} & $a_1a_2^2b_2a_1^2a_2b_1a_2^3b_2$ &-1010-10001&1-10\\
4\phantom{*} & $a_1^3b_1a_2^3b_2a_2^3b_3$ &00-1000001&-1-12 \\
\hline
1* & $a_2a_1^2b_1a_2a_1^2b_1a_2a_1^2b_1$  & 10-110-110-1&000\\
2* & $a_2^2a_1b_2a_2a_1^2b_2a_1^3b_1$  & 01-110000-1&01-1\\
3* & $a_2a_1^2b_1a_2^2a_1b_2a_1^3b_1$  & 10-101000-1&-110\\
4* & $a_2^3b_3a_1^3b_2a_1^3b_1$ & 00100000-1&11-2\\
\hline
5\phantom{*} & $a_2^3b_1a_3a_2^2b_2a_1a_2^2b_1$&0000-12-100&-110\\
5* & $a_2^3b_2a_1a_2^2b_1a_3a_2^2b_2$  & 00001-2100&1-10\\
\hline
6\phantom{*} & $a_2^3b_1a_2^3b_1a_3a_2a_1b_1$&00000-111-1&000\\
6* & $a_2^3b_2a_2^3b_2a_1a_2a_3b_2$ & 000001-1-11&000\\
\hline
\end{tabular}}
\end{table}

\begin{Lem}\label{V-intersections}
Let $S\subset X\sm S_1$ be a section represented by one of the $240$ monodromy factorizations from Table \ref{tab:tet}
for the system of arcs shown in Figure \ref{fig:sing-scheme}. 
Then
the sequences of intersection indices $S\circ\V_{p_i}$, $i=1,\dots,9$, and 
$S\circ\V_{q_j}$, $j=1,2,3$,
coincide, respectively, with the $\D a$ and $\D b$-sequence of the factorization.
\end{Lem}

\begin{proof}
Due to Lemma \ref{int-ind},
it is enough to check that $\ind(q,a_i-a_j)=i-j$, where $q=S\cap \T$,
while $\ind(q,b_i-b_j)=j-i$,
which is straightforward from the definition of the curves $a_i, b_i$ (see the figure accompanying Table~\ref{tab:tet}).
\end{proof}

\begin{Prop}\label{homol-distinct-d=1}
The $240$ 
monodromy factorizations listed in Table \ref{tab:d=1} represent $240$
sections $S\subset X\sm S_1$ whose homology classes $[S]\in H_2(X)$ are pairwise distinct.
\end{Prop}

\begin{proof}
Note that two $a$-sequences produce the same $\D a$-sequence only if their index-sequences differ by a constant. 
Therefore, since all $240$ $a$-sequences are distinct and all contain the letter $a_1$ at least once, the $240$ $\D a$-sequences must also all be distinct.
Lemma \ref{V-intersections} then implies that all $240$ sections $S$ represented by factorizations in Table  \ref{tab:d=1}
 are homologically distinguished by their intersection indices $S\circ\V_p$ with the corresponding matching cycles $\V_p$.
\end{proof}

\subsection{Intersection of sections with matching cycles in $X\sm(S_1\cup S_2)$}
The factorizations in Table \ref{tab:d=2} are liftings of $(a_1a_2^2b_1)^3$ from $\Mod(\T_2)$ to $\Mod(\T_3)$. Accordingly, we split the sequence of letters in each factorization into
three subsequences:
\begin{itemize}
\item
a {\it $b$-sequence}  formed by the $4$th, $8$th and $12$th factors (sent to $b_1$ by the forgetful map),
\item
an {\it $a_1$-sequence}  formed by the $1$st, $5$th and $9$th factors (sent to $a_1$), and
\item
an {\it $a_2$-sequence}  formed by the remaining six factors (sent to $a_2$).
\end{itemize}

A comparison of consecutive terms in each subsequence give difference sequences, which we denote by $\D b$, $\D a_1$ and $\D a_2$, respectively.
Namely, 
\begin{itemize}\item
the $b$-sequence $b_{i_1}b_{i_2}b_{i_3}$ gives the $\D b$ sequence $i_1-i_2$, $i_2-i_3$, $i_3-i_1$,
\item
in the $a_1$-sequence, consecutive pairs of letters $a_1\widecheck a_1$ and  $\widecheck a_1 a_1$
give $\D a_1$-index $1$ and $-1$ respectively, while $a_1a_1$ and $\widecheck a_1\widecheck a_1$ give $0$,
\item
in the $a_2$-sequence, consecutive pairs $a_3a_2$ and $a_2\widecheck a_2$ give $\D a_1$-index $1$,
pairs $a_2a_3$ and  $\widecheck a_2 a_2$ give $-1$, and pairs of equal letters give $0$.
\end{itemize}

Table \ref{tab:d2-chain} shows the $\D$-sequences for the sample factorizations of Types 1--9.
As before, pairs of consecutive terms in each subsequence represent vertices connected by edges 
in the figure accompanying Table \ref{tab:d2-chain}.

  \begin{table}[h!]
\caption{$\D$-sequences of sample factorizations of sections $S\subset X\sm(S_1\cup S_2)$. Accompanying figure: Matching cycles in $X\sm(S_1\cup S_2)$.}
\label{fig:E7-graph}
\label{tab:d2-chain}
\raisebox{-22mm}{\includegraphics[alt={As in the figure accompanying Table 4, we construct matching 2-cycles in ${\mathbb T}_3$ formed by edges between vertices having the same labels among the vertices of a 12-gon labelled by $(a_1a_2^2b_1)^3$. Labels $a_2$ produce a hexagon and each of the letters $a_1$
and $b_1$ a triangle.
},height=4.7cm]{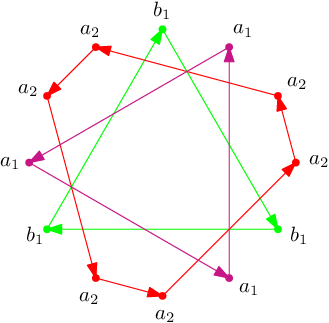}}
\scalebox{0.95}{\begin{tabular}{cccccc}
\hline
Type  & Sample factorization &  $\D a_1$&$\D b$& $\D a_2$\\
\hline
1 & $a_1a_2a_3b_1a_1a_2a_3b_1a_1a_2a_3b_1$ &000&000&-11-11-11 \\
2 & $a_1a_2^2b_1a_1a_2a_3b_1a_1a_3^2b_2$ & 000&0-11&00-1001\\
3 & $a_1a_2a_3b_2a_1a_2^2b_1a_1a_3^2b_2$ &000&1-10&-110-101 \\
\hline
4 & $a_1a_2^2b_1a_1a_2a_3b_2\widecheck a_1a_2^2b_1$ &01-1&-110&00-1100 \\
5 & $a_1a_2a_3b_2a_1a_2^2b_2\widecheck a_1a_2^2b_1$ &01-1&01-1&-110000\\
6 & $a_1a_2^2b_2a_1a_2^2b_2\widecheck a_1a_2a_3b_2$ &01-1&000&0000-11\\
\hline
7 & $\widecheck a_1\widecheck a_2a_2b_1a_1a_2^2b_2\widecheck a_1a_2^2b_2$ &-110&-101&-100001\\
8 & $\widecheck a_1a_2^2b_1a_1a_2^2b_2\widecheck a_1\widecheck a_2a_2b_1$ &-110&-110&0001-10\\
9 & $\widecheck a_1a_2^2b_1a_1a_2\widecheck a_2b_1\widecheck a_1a_2^2b_1$ &-110&000&001-100\\
\hline
\end{tabular}}
\end{table}

\begin{Lem}\label{VpS-d=2}
Let $S\subset X\sm(S_1\cup S_2)$ be a section represented by one of the $56$ monodromy factorizations from Table \ref{tab:d2-chain}
for the system of arcs in Figure \ref{fig:sing-scheme}. 
Then for each edge $p$ in the figure accompanying Table \ref{tab:d2-chain},
the intersection index 
$S\circ\V_{p}$ coincides with the difference index
in the associated $\D$-sequence
corresponding to the endpoints of $p$.
\end{Lem}

\begin{proof}
As in the proof of Lemma \ref{V-intersections},
 it is enough to apply
Lemma \ref{int-ind} for the calculation $S\circ \V_p=\ind(q,\del D_1-\del D_0)$, with
$q=S\cap\T$.
\end{proof}

\begin{Prop}\label{homol-distinct-d=2}
The 56 monodromy factorizations listed in Table \ref{tab:d=2}
represent 56
sections $S\subset X\sm (S_1\cup S_2)$ whose homology classes $[S]\in H_2(X)$ are pairwise distinct.
\end{Prop}

\begin{proof}
The cyclic order of indices in $\D a_1$ distinguishes three groups of
monodromy factorizations 
in Table \ref{tab:d=2} (reproduced in Table \ref{tab:d2-chain}): those of Types 1--3, 4--6 and 7--9.
Combining $\D a_1$ with $\D b$, we can distinguish all nine types from each other.

To distinguish variations of each type from each other (eight variations for Type~1 and six for other types),
it is enough to use $\D a_2$. 
Interpreting $\D$-sequences as intersection indices via Lemma \ref{VpS-d=2},
we conclude that all $56$ sections are homologically distinct.
\end{proof}

\subsection{Chains of matching cycles in $X\sm(S_1\cup S_2\cup S_3)$}
The factorizations in Table \ref{tab:d=3} are liftings of $(a_1a_2a_3b_1)^3$. Correspondingly, we split each 
factorization sequence into four subsequences each of length $3$:
a {\it $b$-sequence} consisting of the $4$th, $8$th and $12$th terms, and $a_i$-sequences, $1\leq i\leq 3$, consisting of the $i$th, $(4+i)$th and $(8+i)$th terms.
In the figure accompanying Table \ref{tab:d3-chain}, we sketched the graph obtained by connecting by edges the points of $\D\subset S^2$ 
representing consecutive terms of each subsequence.
As in the previous case, we define four difference sequences: $\D a_i$, $1\leq i\leq 3$, and $\D b$, whose terms correspond to the edges
of this graph. Namely, for the edge connecting vertices representing terms $x$ and $y$ in the factorization, the difference index in $\D b$
is defined as $i-j$ if $xy=b_ib_j$, and in the $\D a_i$-sequences, $1\leq i\leq 3$, it is defined as
\begin{itemize}\item
$0$,  if $x=y$,
\item
$1$,  if  $xy$   is $a_1\widecheck a_1$, $a_2\widecheck a_2$, $\widehat a_3a_3$,  or $\widehat a_2a_2$,
\item
$-1$,   if $xy$  is $\widecheck a_1 a_1$, $\widecheck a_2 a_2$, $a_3\widehat a_3$, or $a_2\widehat a_2$.      
\end{itemize}

Examples for the sample factorizations are shown in Table \ref{tab:d3-chain}.

\begin{table}[h!]
\caption{$\D$-sequences of sample factorizations of sections $S\subset X\sm(S_1\cup S_2\cup S_3)$. Accompanying figure: Matching cycles in $X\sm(S_1\cup S_2\cup S_3)$.}	
\label{tab:d3-chain}
\label{fig:E6-graph}
\raisebox{-23mm}{\includegraphics[alt={In a way analogous to the figures accompanying Tables 4 and 5, we produce matching 2-cycles corresponding to the labelling
$(a_1a_2a_3b_1)^3$. The Figure shows 4 triangles corresponding to the letters $a_1,a_2,a_3$ and $b_1$.
},height=4.8cm]{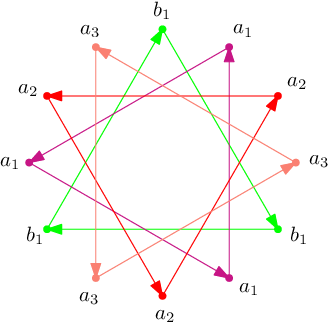}}
\scalebox{0.95}{\begin{tabular}{cccccc}
\hline
Type  & Sample factorization &$\D b$&$\D a_1$&$\D a_2$&$\D a_3$\\
\hline
1 & $a_1a_2\widehat a_3b_1a_1a_2\widehat a_3b_1a_1\widehat a_2a_3b_1$ &000&000&0-11&01-1 \\
2 & $a_1a_2\widehat a_3b_1a_1\widehat a_2\widehat a_3b_2a_1a_2a_3b_1$ &-110&000&-110&01-1 \\
3 & $a_1\widehat a_2\widehat a_3b_2a_1a_2\widehat a_3b_2a_1a_2a_3b_1$ &01-1&000&10-1&01-1 \\
\hline
4 & $a_1a_2\widehat a_3b_2\widecheck a_1a_2a_3b_1a_1a_2a_3b_1$ &10-1&1-10&000&10-1\\
5 & $a_1a_2a_3b_2\widecheck a_1a_2a_3b_1a_1a_2\widehat a_3b_2$ & 1-10&1-10&000&-110\\
6 & $a_1a_2a_3b_2\widecheck a_1a_2\widehat a_3b_2a_1a_2a_3b_2$ & 000&1-10&000&-110\\
\hline
7 & $a_1a_2a_3b_2\widecheck a_1a_2a_3b_2\widecheck a_1\widecheck a_2a_3b_1$ &01-1&10-1&01-1&000 \\
8 & $a_1a_2a_3b_2\widecheck a_1\widecheck a_2a_3b_1\widecheck a_1a_2a_3b_1$ & 10-1&10-1&1-10&000\\
9 & $a_1\widecheck a_2a_3b_1\widecheck a_1a_2a_3b_1\widecheck a_1a_2a_3b_1$ & 000&10-1&-101&000\\
\hline
\end{tabular}}
\end{table}

The proof of the following Lemma is analogous to those of Lemmas \ref{V-intersections} and \ref{VpS-d=2}.

\begin{Lem}\label{VpS-d=3}
Let $S\subset X\sm(S_1\cup S_2\cup S_3)$ be a section represented by one of the $27$ monodromy factorizations from Table \ref{tab:d2-chain}
for the standard system of arcs in Figure \ref{fig:sing-scheme}. 
Then for each edge $p$ in the figure accompanying Table \ref{tab:d3-chain},
the intersection index 
$S\circ\V_{p}$ coincides with the corresponding difference index
in the associated $\D$-sequence 
corresponding to the endpoints of $p$.
\qed\end{Lem}

\begin{Prop}\label{homol-distinct-d=3}
The $27$ monodromy factorizations listed in Table \ref{tab:d=3} represent $27$
sections $S\subset X\sm (S_1\cup S_2\cup S_3)$ whose 
homology classes $[S]\in H_2(X)$ are  pairwise distinct.
\end{Prop}

\begin{proof}
The Types 1-9 of factorizations are distinguished by the pair of sequences (viewed up to cyclic permutations) $\D b$ and $\D a_1$.
Each type has three shift-variations, and 
to distinguish them from each other we can use any $\D a_i$-sequence different from $000$:
for instance, $\D a_2$ for Types 1, 2 and 3, and $\D a_1$ for the other types.
The proof is completed by appealing to Lemma \ref{VpS-d=3}.
\end{proof}

\section{Straightening sections}\label{Schevchishin}
For generalities on $J$-holomorphic curves we refer the reader to \cite{MS} or \cite{W}.

\subsection{Proof of Theorem \ref{sections-are-lines}}\label{Sec:Schevchishin}
Let $J_0$ be the given integrable almost complex structure on $X$.
Choose a small neighbourhood $U\subset X\sm(S_1\cup\dots\cup S_d)$ of the $12$-point critical locus $\Crit\subset X$ of $f$.
Denote by $\J$ the space of almost complex structures  $J\colon TX\to TX$ such that
\begin{enumerate}\item
$J$ coincides with $J_0$ inside $U$,
\item
the fibres of $f$ are pseudoholomorphic.
\end{enumerate}
Clearly these conditions imply that
the orientations on $X$ and on the fibres of $f$ induced by $J$ and $J_0$ coincide.

\begin{Lem}\label{tame}
For each $J\in\J$ there exists a symplectic structure $\om$ on $X$ which tames $J$.
\end{Lem}

\begin{proof}
The taming structure is $\om_0+Mf^*\om_{S^2}$ for a sufficiently large constant $M$, where $\om_0$
comes from the K\"ahler structure for the initial complex algebraic structure on $X$, and $\om_{S^2}$
is a symplectic form on $S^2$.
\end{proof}

\begin{Lem}\label{connected}
$\J$ is a path-connected space.
\end{Lem}

\begin{proof} There are no obstructions to connecting a pair of structures in $\J$ since conditions (1) and (2) imply contractibility of the space of corresponding structures $J_x$ on $T_xX$ for all points $x\in X$.
\footnote{
This contractible space is just the quotient of the subgroup of $GL_4(\R)$ formed by $2\!\!\times\!\!2$-block matrices 
$\left[ \begin{array}{c|c}
   A & B \\
   \midrule
   0 & C \\
\end{array}\right]$ with
$\det A,\det C>0$
by subgroup of matrices
$\left[\begin{matrix} a&b\\0&c\end{matrix}\right]\in GL_2(\C)
$.}
\end{proof}

\begin{Lem}\label{exist}
	For any given disjoint sections $S_i$, $i=1,\ldots,d$,
there exists an almost complex structure $J\in\J$ 
such that the $S_i$ are $J$-holomorphic.
\end{Lem}

\begin{proof}
One can construct a $J$ satisfying (1) and (2) in a neighbourhood of the $S_i$ that makes $S_i$ $J$-holomorphic
and then extend it to $X$ using the absence of obstructions (see the proof of previous lemma).
\end{proof}

Define $\JS$ to be the set of pairs $(J,\SS)$, where $J\in\J$ and $\SS=\{\S_1,\dots,\S_d\}$ is a set of $J$-holomorphic sections such that $[\S_i]=[S_i]\in H_2(X)$ for each $i$.
Denote by $p\colon \JS\to\J$ the projection $(J,\SS)\mapsto J$.

\begin{Lem}\label{homeomorphism}
The projection $p$ is a homeomorphism. 
\end{Lem}

\begin{proof}
Notice that $p^{-1}(J)$ cannot contain more than one element, since $\S_i^2=S_i^2<0$ for all $i$.
On the other hand, for each $i$, the class $[S_i]$ must contain a possibly reducible $J$-holomorphic curve $\S_i$ by the
Gromov--McDuff Theorem \cite[Lemma~3.1]{MD}. If $\S_i$ has more than one irreducible component, then, since $c_1(X)[\S_i]=1$, 
at least one component of $\S_i$, say $\S_i^*$, must have $c_1(X)[\S_i^*]<1$. Since the Poincar\'e dual of $c_1(X)$ is represented 
by a fibre of $X$, which also is $J$-holomorphic, by positivity of intersections, $\S_i^*$ must be contained in a fibre.
It follows that $\S_i^*$ must coincide with (or be a multiple cover of) a singular fibre.
Such a possibility can be ruled out by arguing as in \cite[Prop.~4.1]{OO}.
Thus $\S_i$ is irreducible and, by the adjunction formula, is embedded.
By positivity of intersections again, $\S_i$ is a section. 
\end{proof}

\begin{proof}[Proof of Theorem \ref{sections-are-lines}]
We can connect $J$ from Lemma \ref{exist} with $J_0$ by a path using Lemma \ref{connected}, and 
$p^{-1}$, by lemma \ref{homeomorphism}, connects $\{S_1,\dots,S_i\}$ with some set of disjoint lines in $X$.
\end{proof}

\subsection{Isotopy of sections in $X\sm(S_1\cup\dots\cup S_d)$}

\begin{Thm}\label{sections-fixed}
Let $S_1,\dots,S_d$ be disjoint sections in $X$ and $S,S'$ be homologous sections in $X^*=X\sm (S_1\cup\dots\cup S_d)$.
Then there exists an isotopy  between $S$ and $S'$ in $X^*$.
\end{Thm}

\begin{proof}
The scheme of arguments is the same as in the proof of Theorem \ref{sections-are-lines}.
Namely, we consider the subset $\J(S_1,\dots,S_d)\subset\J$ formed by almost complex structures satisfying in addition to conditions (1) and (2)
the condition 

(3) The sections $S_1,\dots,S_d$ are pseudoholomorphic.

Path-connectedness of $\J(S_1,\dots,S_d)\subset\J$ is proved in precisely the same way as in Lemma \ref{connected}.
Then we similarly define 
 $\JS(S_1,\dots,S_d)$ to be the set of pairs $(J,\S)$, where $J$ is in $\J(S_1,\dots,S_d)$ and 
 $\S\subset X\sm(S_1\cup\dots\cup S_d)$ is a J-holomorphic section whose class $[\S]\in H_2(X)$ agrees with $[S]=[S']$.
The projection 
$$p\colon\JS(S_1,\dots,S_d))\to\J(S_1,\dots,S_d),\quad (J,\S)\mapsto J$$
is a homeomorphism for the same reasons as in Lemma \ref{homeomorphism}.
We then complete the proof of Theorem \ref{sections-fixed} using the existence of structures $J,J'\in\J(S_1,\dots,S_d)$ for which
sections $S$ and, respectively, $S'$ are pseudoholomorphic.
\end{proof}

\section{Proof of the main Theorems}
\subsection{Proof of Theorem \ref{universality} }
By \cite{M} (see also Section \ref{sec:elliptic}), $X$ and $X'$ as LFs can be assumed to be algebraic
(rational elliptic surfaces).
Using Theorem \ref{sections-are-lines}, we reduce the setting to the case of families of lines $S_1,\dots,S_m\subset X$ and $S_1',\dots,S_m'\subset X'$.
Blowing down these lines give del Pezzo surfaces $Y$ and $Y'$ of degree $m$ endowed with marked points resulting from blowing down.
The marked points cannot lie on any of the lines on $Y$ and $Y'$ because of the assumption that $X$ and $X'$ have only fishtail singular
fibres and hence can have no curves of self-intersection less than $-1$.
One can then blow down a further set of $9-m$ disjoint lines on
$Y$ and on $Y'$ to obtain a projective plane with configurations of points $Q=(q_1,\dots,q_9)$ and $Q'=(q_1',\dots,q_9')$, respectively
(excluding the case $Y\cong Y'\cong\CP1\times\CP1$).

For a $9$-tuple that gives an elliptic surface after blowing up, the first $8$ points can be an arbitrary
generic $8$-tuple (no $3$ points are collinear, no $6$ are conconic, and not all $8$ lie on a cubic with a node at one of them),
and the $9$th point is determined by the first $8$. Thus the set of such $9$-tuples is Zariski-open in $(\cp)^8$ and, in particular, is path-connected.
A path connecting the $9$-tuples $Q$ and $Q'$ induces an isomorphism of $X$ and $X'$ as LFs.
In the exceptional case $Y\cong Y'\cong\CP1\times\CP1$, the argument is analogous.

\subsection{Proof of Theorem \ref{Bijection}} 
Injectivity follows immediately from Theorem~\ref{sections-fixed}.

To prove surjectivity, we apply Theorem \ref{sections-are-lines}
to isotopically deform sections $S_i$ into some lines $L_i$, $i=1,\dots,d$, and then extend this isotopy to an ambient isotopy
of $X$ through the strong automorphisms (cf., \cite[Chapter~8, Theorem~1.3]{H}).
This results in a strong automorphism
$\Phi\:X\to X$, with $\Phi(S_i)=L_i$.
The given class $h$ is realized by some line $L\subset X$ by Proposition \ref{classes-of-sections}, which is disjoint from $L_i$ since $h\cdot S_i=L\cdot L_i=0$.
Then, $S=\Phi^{-1}(L)$ is a required section.
\qed

\subsection{Proof of Theorem \ref{factorization-general}}
Due to Theorem \ref{factorization-conjugacy}, it is enough to show that two sections $S$ and $S'$ in $X^*$ are isotopic if and only if
there exists a strong automorphism $\Phi\colon X\to X$ preserving sections $S_i$, $i=1,\dots,d$, and sending $S$ to $S'$.
First, an isotopy of $S$ and $S'$ can be extended to an ambient isotopy of $X$ through strong automorphisms supported 
in $X^*$ (in fact, in any neighbourhood of the trace of the isotopy between the sections). This isotopy results in a strong automorphism $\Phi$ as required.

In the opposite direction, the existence of such an automorphism $\Phi$ implies that $[S]=[S']\in H_2(X)$, by Proposition \ref{fiberwise-automorphism}.
Then, by Theorem \ref{sections-fixed}, $S$ and $S'$ are isotopic.
\qed

\subsection{Proof of Theorem \ref{Main}}
It follows from Theorems \ref{Bijection} and \ref{factorization-general} that there exist precisely $n_d$ 
($n_1=240$, $n_2=56$, $n_3=27$)
lifts of a given factorization of $(X;S_1,\dots, S_d)$.
Propositions \ref{factorization1+1}, \ref{factorization1+2} and \ref{factorization1+3}
provide the required number, $n_d$, of sections described by the monodromy factorizations
in the Tables \ref{tab:roots}, \ref{tab:d2-chain} and \ref{tab:d3-chain},
while
Propositions \ref{homol-distinct-d=1}, \ref{homol-distinct-d=2} and \ref{homol-distinct-d=3}
guarantee that 
these $n_d$ sections are homologically distinct; therefore, the specified lists of isotopy classes of sections are complete.
\qed

\subsection{Proof of Theorem \ref{Main-Th}}
Since del Pezzo surfaces $Y$ of a fixed degree $d=1,2,3$ are deformation equivalent, it is sufficient to consider one example obtained by contraction of
$d$ disjoint lines on an elliptic surface.

It follows from
Theorems \ref{sections-are-lines} and \ref{Main} that if we fix a set of disjoint lines
$L_1,\dots,L_d$ on an elliptic surface $X$ so that they represent the basic monodromy factorization
$(a^3b)^3$ for $d=1$, $(a_1a_2^2b_1)^3$ for $d=2$ and $(a_1a_2a_3b_1)^3$ for $d=3$,
then the set of $n_d$ lines in $X\sm(L_1\cup\dots\cup L_d)$ will be represented by the lifts
of the basic factorizations listed in the
Tables \ref{tab:roots}, \ref{tab:d2-chain} and \ref{tab:d3-chain}.
 Then contraction of the lines $L_1,\dots,L_d$ will produce a del Pezzo surface $Y$ whose lines $L\subset Y$
have monodromy factorizations from these tables.\
\qed

\section{Monodromy factorizations and their presentation by $E_8$-roots}\label{root-section}

\subsection{The graphs $E_8$, $E_7$  and $E_6$ formed by matching cycles}\label{2-roots}
Figure \ref{fig:root-graphs} shows subgraphs of the graphs in the figures accompanying Tables \ref{fig:E8-graph}, \ref{fig:E7-graph} and \ref{fig:E6-graph}
whose edges $p$ (other than the dashed ones) define systems of matching cycles $\V_p$ that have intersection schemes $E_8$, $E_7$ and $E_6$, respectively.
If we add the dashed edge to each graph, then the intersection schemes extend to $\til E_8$, $\til E_7$ and $\til E_6$.

\begin{figure}[h!]
\includegraphics[alt={On the left, twelve vertices are labelled $(a^3b)^3$ around a circle. The $a$'s are
connected by arrows from the one o'clock position to the eleven o'clock position with the last
edge dashed. The first two $b$'s are also connected by an arrow. In the middle, a representation of $(a_1a_2b_1)^3$ is shown. Consecutive $a_2$'s are connected (except the first and last), as are the first two $b_1$'s, and two $a_1$ are connected by an arrow and two by a dashed arrow. On the right, $(a_1a_2a_3b_1)^3$ is shown with vertices having the same labels connected depicting $E_6/\widetilde E_6$.
},height=5cm]{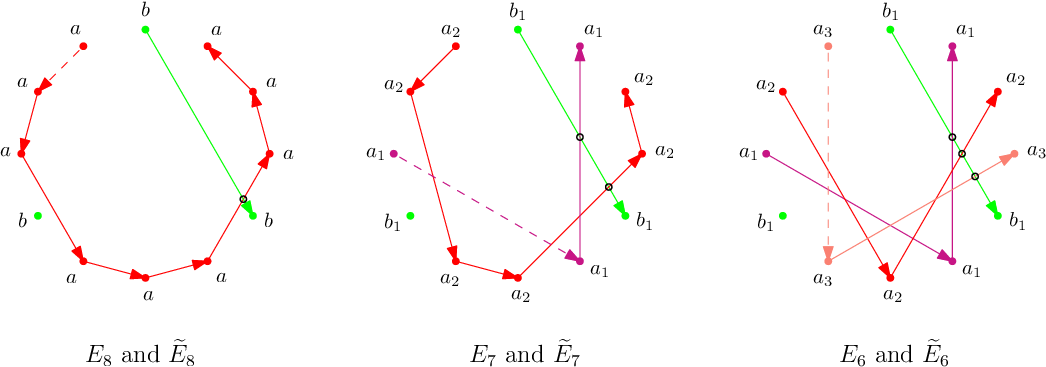}
\caption{\footnotesize Intersection graphs $E_8,E_7,E_6$ formed by matching cycles.}
\label{fig:root-graphs}
\end{figure}

In the case of $E_8$ (the first graph), the intersection scheme is evident.
For the other two graphs we need to notice that only the encircled intersections of edges give intersecting matching cycles 
(since vanishing 1-cycles over the other intersection points can be made disjoint).

It is also apparent that $X$ splits into a union of a
plumbing $E_8$-manifold
(a regular neighbourhood, $N$, of the union of matching cycles $\V_p$ for the solid edges
of the leftmost graph), and the complementary {\it nucleus} plumbed from a neighbourhood of the section $S_1\subset X$
and a cusp-fibre neighbourhood.
The latter is a neighbourhood of the
union of a nonsingular fibre $\T$ and a pair of Lefschetz thimbles $D_1,D_2$
attached to $\T$ so that the vanishing 1-cycles $\del D_1,\del D_2\subset\T$ intersect in one point.
In Figure \ref{fig:root-graphs}, such discs $D_i$ are associated to the pair of singular values corresponding to the isolated
vertices of the $E_8$-graph.

It follows that the matching cycles over the edges of this $E_8$-graph 
represent basic roots of
$E_{8}=H_2(N)=\la \T,S_1\ra^\perp\subset H_2(X).$
With the edges directed as in Figure \ref{fig:root-graphs}, the matching cycles will to intersect each other positively with respect to the induced orientations.

\subsection{Coordinates in $E_8$ and $\til E_8$, and their duals}
In the $E_8$-lattice, we will enumerate a set $\{e_0,\dots,e_7\}$ of basic roots 
so that their product is $1$ or $0$ according to whether they are adjacent or not in the
 $E_8$-diagram \ \ $\begin{matrix}e_1e_2e_3e_4e_5e_6e_7\\ e_0\phantom{AA}\end{matrix}$. 
We will also arrange 
the coordinates $y,x_i\in\Z$ of an element $v=y e_0+x_1e_1+\dots+x_7e_7\in E_8$ in the form of a matrix-diagram
$[v]_{E_8}=\begin{matrix}{x_1x_2}x_3{x_4x_5}{x_6x_7}\\ y\phantom{\aa\aa\aa}\end{matrix}$.

Elements $\ell\in E_8^*=\Hom(E_8,\Z)$ of the dual lattice will be presented similarly as matrices
$[\ell]_{E_8^*}=\begin{matrix}{x_1x_2}x_3{x_4x_5}{x_6x_7}\\ y\phantom{\aa\aa\aa}\end{matrix}$, where
$\ell(e_0)=y$ and $\ell(e_i)=x_i,1\leq i\leq 7$.

Let $\Psi\colon E_8\to E_8^*$ denote the
duality isomorpism, where $\Psi(v)$ sends $w\in E_8$ to $v\cdot w$.
The following Proposition tells how the root $e=[S_2]-[S_1]-[\T]$ (see Corollary \ref{sec-roots}) can be determined from a monodromy factorization 
in Table \ref{tab:d=1}.

\begin{Prop}\label{E8*-coefficients}
Assume that $S_1,S_2\subset X$ are disjoint  sections, and $\D a=\aa_1,\dots,\aa_9$, $\D b=\bb_1,\bb_2,\bb_3$ are the $\D$-sequences determined by
the monodromy factorization of $S_2$ in $X\sm S_1$ as described in Section \ref{2-cycles}.
Then the dual $\Psi(e)\in E_8^*$ of the root  
$$e=[S_2]-[S_1]-[\T]\in E_8=\la \T,S_1\ra^\perp\subset H_2(X)$$
has coordinates
$[\Psi(e)]_{E_8^*}=\begin{matrix}{\aa_1\aa_2}\aa_3{\aa_4\aa_5}{\aa_6\aa_7}\\ \bb_3\phantom{\aa\aa\aa}\end{matrix}.$
\end{Prop}

\begin{proof} This follows in a straightforward manner
from Corollary \ref{sec-roots} and Lemma \ref{V-intersections}.
\end{proof}

It is apparent that $\Psi(e)$ just expresses the image of the class $[S_2]^*\in H^2(X)$, Poincar\'e dual to $[S_2]$, under the pull-back homomorphism
$H^2(X)\to H^2(N)=E_8^*$. 
Let us extend the $E_8$-plumbing manifold $N$ to an $\til E_8$-plumbing $N'$ by adding a neighbourhood of the matching cycle $\V_p$ for the dashed edge $p$.
Clearly, the basic roots of $H_2(N')=\til E_8$ are obtained from the roots $e_0,\dots,e_7\in H_2(N)=E_8$ by adding
the new cycle $\V_p\subset N'$, which we will denote by $e_8$.
The same convention about presentation
of elements $v\in\til E_8$ and $\ell\in\til E_8^*$ as coordinate matrices $[v]_{\til E_8}$ and $[\ell]_{\til E_8^*}$ will be used.
Then the pull-back of $[S_2]^*$ under the homomorphism $H^2(X)\to H^2(N')=\til E_8^*$ can be described in a similar way to Proposition \ref{E8*-coefficients}.

\begin{Prop}\label{extendedE8*-coefficients}
Under the assumptions of Proposition \ref{E8*-coefficients}, the class $\ell\in H^2(N')=\til E_8^*$ defined by $h\mapsto h\cdot[S_2]$, for $h\in H_2(N')$,
 has $E_8^*$-coordinates
 $[\ell]_{\til E_8^*}=\begin{matrix}{\aa_1\aa_2}\aa_3{\aa_4\aa_5}{\aa_6\aa_7\aa_8}\\ \bb_3\phantom{\aa\aa\aa xx}\end{matrix}$.
\qed\end{Prop}

\subsection{Passing from $\til E_8$- to $E_8$-coordinates}
Let $\imath^*\colon\til E_8^*=H^2(N')\to H^2(N)= E_8^*$ be the map induced by the inclusion $\imath\colon N\subset N'$, and 
$\til\Psi\colon\til E_8\to\til E_8^*$
be the duality homomorphism, that is, $\til\Psi(v)$ sends $h\in \til E_8$ to $v\cdot h$.

\begin{Prop}\label{E8-E8*}
Under the assumptions of Proposition \ref{E8*-coefficients}, let $v\in\til E_8$ be an element with coordinates
$[v]_{\til E_8}=\begin{matrix}{x_1x_2}x_3{x_4x_5}{x_6x_7x_8}\\ y\phantom{\aa\aa XX}\end{matrix}$.
Then the image of $v$ 
under the homomorphism
$$\begin{CD}
\til E_8@>\til\Psi>>\til E_8^*@>\imath^*>> E_8^*@>\Psi^{-1}>> E_8
\end{CD} $$
has coordinates 
$\begin{matrix}{x_1x_2}x_3{x_4x_5}{x_6x_7}\\ y\phantom{\aa\aa\aa}\end{matrix}-x_8 \begin{matrix}{2465432}\\ 3\phantom{aa}\end{matrix}.
$
\end{Prop}

\begin{proof}
This follows immediately from the well-known (and trivial) fact that
the kernel of $\til\Psi$ is spanned by 
$\begin{matrix}{24654321}\\ 3\phantom{aaa}\end{matrix}$.
\end{proof}

\subsection{Root vectors in $E_8$ for monodromy factorizations}\label{root-variations}
After we fix a section $S_1\subset X$, the root-class $e=[S]-[S_1]-[\T]\in E_8=\la S_1,\T\ra^\perp$
for any other section $S\subset X\sm S_1$ (see Corollary \ref{sec-roots}) can be described in terms of the
monodromy factorization of $X$, as shown in Table  \ref{tab:permutation}.
Namely, this Table comprises 120 roots $e$ for $S$ having
sample factorizations of Types  1*--\,6* (together with their variations),
while the other 120 roots represented by factorizations of Types 1--6 correspond to their negatives, that is, $-e$, since Types 1--6 are obtained by reversion of Types 1*--\,6*
(see Section \ref{symmetries}).

To explain the structure of Table \ref{tab:permutation}, consider its first row representing the 27 factorizations of {Type 1*}.
After the $a$-indices and $\D a$-indices of the sample factorization $(a_1a_2^2b_1)^3$, we indicate coordinates
$[\ell]_{\til E_8^*}$ of the image $\ell\in H^2(N')=\til E_8^*$
of $[S]^*\in H^2(X)$ found by Proposition \ref{extendedE8*-coefficients}.

\begin{table}[h!]
\caption{Roots $e\in E_8$ representing
monodromy factorization of sections $S\subset X\sm S_1$}\label{tab:permutation}
\scalebox{0.7}{\begin{tabular}{cccccc}
\hline
 Type &$a$-indices& $\Delta a$-sequence& $[\ell]_{\til E_8^*}$&$[v]_{\til E_8}$& variations of the $[v]_{\til E_8}$-fragments\\
\hline
1* & $211211211$ & 10-110-110-1&$\begin{matrix}\text{10-110-110}  \\  \text{0}\phantom{aaa}\end{matrix}$&
$\begin{matrix}\underline{01}\,2\,\underline{11}\,\underline{100}\\ 1\phantom{aaa}\end{matrix}$&
$\underline{01}\to\begin{cases} 11\\12\end{cases}$\ \ $\underline{11}\to\begin{cases} 21\\22\end{cases}$\ \ $\underline{100}\to\begin{cases} 110\\111\end{cases}$
\\
\hline
2* & $221211111$ & 01-110000-1 &$\begin{matrix}\text{01-110000}  \\  \text{-1}\phantom{aaa}\end{matrix}$&
 $\begin{matrix}\underline{00}1\underline{00}000\\ 1\phantom{AA}\end{matrix}$&
\hskip-3mm$\underline{00}\to\begin{cases} 01\\11\end{cases}$\ \ $\underline{00}\to\begin{cases} 10\\11\end{cases}$  \\
2* & $211111221$ & 10000-101-1 & $\begin{matrix}\text{10000-101}  \\  \text{0}\phantom{aaa\,}\end{matrix}$&
$\begin{matrix}\underline{01}222\underline{210}\\ 1\phantom{AA}\end{matrix}$  &
$\underline{01}\to\begin{cases} 11\\12\end{cases}$\ \ $\underline{210}\to\begin{cases} 211\\221\end{cases}$ \\
2* & $111221211$ & 00-101-1100 & $\begin{matrix}\text{00-101-110}  \\  \text{1}\phantom{aaa\,}\end{matrix}$&
$\begin{matrix}123\underline{21}\,\underline{100}\\ 1\phantom{AA}\end{matrix}$  &
$\underline{21}\to\begin{cases} 22\\32\end{cases}$\ \ $\underline{100}\to\begin{cases} 110\\111\end{cases}$ \\
\hline
3* & $211221111$ & 10-101000-1 & $\begin{matrix}\text{10-101000}  \\  \text{0}\phantom{aaa}\end{matrix}$&
 $\begin{matrix}\underline{01}2\underline{10}000\\ 1\phantom{aaa}\end{matrix}$  &
\hskip-3mm$\underline{01}\to\begin{cases} 11\\12\end{cases}$\ \ $\underline{10}\to\begin{cases} 11\\21\end{cases}$\\
3* & $221111211$ & 01000-110-1 & $\begin{matrix}\text{01000-110}  \\  \text{-1}\phantom{aAA}\end{matrix}$&
$\begin{matrix}\underline{00}111\underline{100}\\ 1\phantom{aaa}\end{matrix}$ &
$\underline{00}\to\begin{cases} 01\\11\end{cases}$\ \ $\underline{100}\to\begin{cases} 110\\111\end{cases}$\\
3* & $111211221$ &  00-110-1010 &$\begin{matrix}\text{00-110-101}  \\  \text{1}\phantom{aaa}\end{matrix}$&
 $\begin{matrix}123\underline{22}\,\underline{210}\\ 1\phantom{aaA}\end{matrix}$& 
$\underline{22}\to\begin{cases} 32\\33\end{cases}$\ \ $\underline{210}\to\begin{cases} 221\\211\end{cases}$\\
\hline
4* & $222111111$ & 00100000-1 & $\begin{matrix}\text{00100000}  \\  \text{-2}\phantom{aaaa}\end{matrix}$&
 $\begin{matrix}00000000\\ 1\phantom{AA}\end{matrix}$ & no $a$-block variations \\
4* & $111111222$ & 00000-1001 & $\begin{matrix}\text{00000{-1}00}  \\  \text{1}\phantom{aaa}\end{matrix}$&
 \text{--}\ $\begin{matrix}12321000\\ 2\phantom{AA}\end{matrix}$ & no $a$-block variations \\
4* & $111222111$ & 00-1001000 & $\begin{matrix}\text{00-100100}  \\  \text{1}\phantom{aaa}\end{matrix}$&
 $\begin{matrix}12321000\\ 1\phantom{AA}\end{matrix}$  & no $a$-block variations \\
\hline
5* & $222122322$ & 001-10-1100 &$\begin{matrix}\text{001{-1}0{-1}10}  \\  0\phantom{aaaa}\end{matrix}$&
 $\begin{matrix}000\underline{11}\,\underline{100}\\ 0\phantom{AA}\end{matrix}$   &
$\underline{11}\to\begin{cases} 00\\01\end{cases}$\ \ $\underline{100}\to\begin{cases} 110\\111\end{cases}$\\
5* & $122322222$ & -10-1100001 & $\begin{matrix}\text{-10-110000}  \\  \text{1}\phantom{AA}\end{matrix}$&
 $\begin{matrix}\underline{11}1\underline{00}000\\ 0\phantom{aaa}\end{matrix}$  &
\hskip-2mm$\underline{11}\to\begin{cases} 00\\01\end{cases}$\ \ $\underline{00}\to\begin{cases} 10\\11\end{cases}$\\
5* & $322222122$ & 100001-10-1 & $\begin{matrix}\text{100001-10}  \\  \text{-1}\phantom{AAa}\end{matrix}$&
\text{--}\  $\begin{matrix}\underline{11}111\underline{100}\\ 0\phantom{AA}\end{matrix}$  &
$\underline{11}\to\begin{cases} 00\\01\end{cases}$\ \ $\underline{100}\to\begin{cases} 110\\111\end{cases}$\\
\hline
6* & $222222123$ & 000001-1-11 & $\begin{matrix}\text{000001-1-1}  \\  \text{0}\phantom{AAA}\end{matrix}$&
 $\begin{matrix}00000\underline{011}\\ 0\phantom{AA}\end{matrix}$  &
$\underline{011}\to\begin{cases} 001\\010\end{cases}$\\
6* & $222123222$ & 001-1-11000 & $\begin{matrix}\text{001-1-1100}  \\  \text{0}\phantom{aaaa}\end{matrix}$&
$\begin{matrix}000\underline{11}000\\ 0\phantom{AA}\end{matrix}$   &
$\underline{11}\to\begin{cases} 01\\10\end{cases}$\\
6* & $123222222$ & -1-11000001 & $\begin{matrix}\text{-1-1100000}  \\  \text{0}\phantom{A}\end{matrix}$&
 $\begin{matrix}\underline{11}000000\\ 0\phantom{AA}\end{matrix}$  &
$\underline{11}\to\begin{cases} 01\\10\end{cases}$\\
\hline
\end{tabular}}
\end{table}
Next, we find coordinates $[v]_{\til E_8}=\begin{matrix}\underline{01}2\underline{11}\,\underline{100}\\ 1\phantom{aaa}\end{matrix}$
of a vector $v\in \til\Psi^{-1}(\ell)$.
By Proposition \ref{E8-E8*}, the coordinates $[e]_{E_8}=\begin{matrix}{01}2{11}{10}\\ 1\phantom{aa}\end{matrix}$
of the required root $e\in E_8$ are obtained in this case 
by just dropping the last zero coefficient of $[v]_{\til E_8}$.

The last column indicates how the coordinate vector $[v]_{\til E_8}$ changes under block permutations of the given sample factorization.
We underlined three fragments, "$01$", "$11$" and "$100$" in $[v]_{\til E_8}$, which correspond to the three consecutive
 $a$-blocks, $a_2a_1^2$. If the first $a$-block is permuted to
 $a_1a_2a_1$ or $a_1^2a_2$, then the corresponding fragment  
$01$ is replaced by $11$ or $12$, respectively.
Similarly, changing the second $a$-block leads to the replacement of the fragment $11$ by $21$ or $22$, respectively,
and changing the third $a$-block replaces the fragment $100$ by $110$ or $111$.

Since all three $a$-blocks for the Type 1* sample factorization are the same, shift-variations have no effect.
In contrast, shift-variations for Types 2*--\,6* sample factorizations lead to non-trivial cyclic permutations of the $a$-indices, and 
each of these types occupy three rows in Table \ref{tab:permutation} correspondingly.

\subsection{The lattice vectors representing factorizations in Table \ref{tab:d=2} and \ref{tab:d=3}}
The same method as was described for Table \ref{tab:d=1} can be applied for Tables \ref{tab:d=2} and \ref{tab:d=3} to describe
the corresponding vectors  in $E_7^*$ and $E_6^*$.
For instance, with reference to the $E_7$-diagram of edges in Figure \ref{fig:root-graphs} and applying 
Lemma \ref{VpS-d=2},
 the $\D$-sequences for the Type 1 sample factorization in Table \ref{tab:d=2} give
the $E_7^*$-vector 
$\begin{matrix}\text{001-11-1}\\ \text{-1}\phantom{aaa}\end{matrix}$.
Here the chain -11-11-1 comes from $\D a_2$, one zero from $\D b$ and the other from $\D a_1$.
The dual vector 
is $\displaystyle \frac{v}2$, where  $v=\begin{matrix}\text{000101}\\ \text{1}\phantom{\ }\end{matrix}\in E_7$.

Similarly, the Type 1 sample factorization in Table \ref{tab:d=3} leads to the
$E_6^*$-vector 
$\begin{matrix}\text{0000-1}\\ \text{0}\phantom{\ }\end{matrix}$, whose dual is
$\displaystyle \frac{v}3$, where $v=\begin{matrix}\text{24654}\\ \text{3}\end{matrix}\in E_6$.

\section{Concluding remarks on the Hurwitz moves' action}\label{concluding}

\subsection{The action of Hurwitz moves (half-twists) on monodromy factorizations}
The monodromy factorizations in each of the Tables \ref{tab:d=1}--\ref{tab:d=3} are related to each other via Hurwitz moves.
We explain how this works in the case of Table \ref{tab:d=1}.
Given an arc $p$ connecting a pair of points in $\D\subset S^2$, we 
denote by $\s_p\in\Mod(S^2,\D)$ the mapping class of the {\it half-twist} along the arc $p$ (for terminology, see \cite[Chapter 9]{FM}).

\begin{Prop}\label{covering-Dehn}
Assume that $f\colon X\to S^2$ is an LF of type $E(1)$, $S_0\subset X$ a section,
 $p\subset S^2$ a matching arc and 
$\phi\colon (S^2,\D)\to(S^2,\D)$ a diffeomorphism representing $\s_p$.
Then 

(1) $\phi$ is covered by an automorphism $F\colon X\to X$ preserving $S_0$, 

(2)
the induced map $F_*\colon H_2(X)\to H_2(X)$ is the Picard-Lefschetz transformation
$$\quad\rho_e(x)=x+(e\cdot x) e,$$ where $e=[\V_p]$.
\end{Prop}

\begin{proof}
A half-twist along a matching arc $p$ does not change the monodromy representation {$\pi_1(S^2\sm\D,\b)\to\PMod(\T_1)$}. This implies 
the existence of the automorphism $F\colon X\to X$ in part (1).

Moreover, $F$ may be localized along the $(-2)$-sphere $\V_p$
(i.e., chosen to be the identity outside a small neighbourhood of $\V_p$), which implies that $F$ has to be a Dehn twist
along $\V_p$, since it reverses the orientation of $\V_p$. {(See, e.g., \cite{S} for details of Dehn twists along $(-2)$-spheres.)}
In fact, the localization of $F$ can be seen as an automorphism of an LF projecting a neighbourhood of $\V_p$
over a disc-neighbourhood of the arc $p$. 
Such an automorphism is known to be a Dehn twist along the $(-2)$-sphere
covering a half-twist  along $p$.
\end{proof}

This gives the following method to relate any pair of the $240$ monodromy factorizations indicated in Table \ref{tab:d=1}
in terms of the corresponding roots  $e,e'\in E_8$.
Assume first that  $e$ and $e'$ are neighbours in the Hasse diagram of $E_8$-roots, namely,
$e'=e+ e_i$,  $e\cdot e_i=1$, where $e_i\in E_8$ is one of the basic roots represented by some matching cycle $p_i$ in Figure \ref{fig:root-graphs}. 
Then $\rho_{e_i}(e)=e'$ and, by Proposition \ref{covering-Dehn},
the factorizations represented by $e'$ and by $e$ are related by
a half-twist along $p_i$.
 In general, roots $e$ and $e'$ are related by adding several basic roots $e_i$ as above, and we need to use several
 half-twists to relate the factorizations associated with $e$ and $e'$.

For example, the factorizations $(a_2a_1^2b_1)^3$ of Type 1* and $a_2a_1a_2b_2a_1^2a_2b_2a_1^3b_1$ of Type 2* are represented
by the $E_8$-roots $e'=\begin{matrix}0121110\\ 1\phantom{aa}\end{matrix}$ and
$e=\begin{matrix}0111100\\ 1\phantom{aa}\end{matrix}$, respectively (see Table \ref{tab:permutation}).
The root $e'$ is obtained from $e$
 by adding two basic roots, 
 $e_3=\begin{matrix}0010000\\ 0\phantom{aa}\end{matrix}$ and 
 $e_6=\begin{matrix}0000010\\ 0\phantom{aa}\end{matrix}$. So the pair of half-twists corresponding 
to $e_3$ and $e_6$ transform any system of arcs representing one of the factorizations into one representing the other;
see Figure~\ref{fig:Hurwitz-moves} (b).
From an algebraic viewpoint, this pair of half-twists look like a pair of substitutions
$a_1b_1a_2=a_2b_2a_1$ relating the given factorizations.

\begin{figure}[h!]
\includegraphics[alt={On the left, the effect of a half-twist along a matching arc is shown on a pair of arcs
connecting the basepoint to the singular values. On the right, the effect of two half-twists is given
on a standard system of arcs, and the corresponding monodromy substitution is shown.
},height=8.5cm]{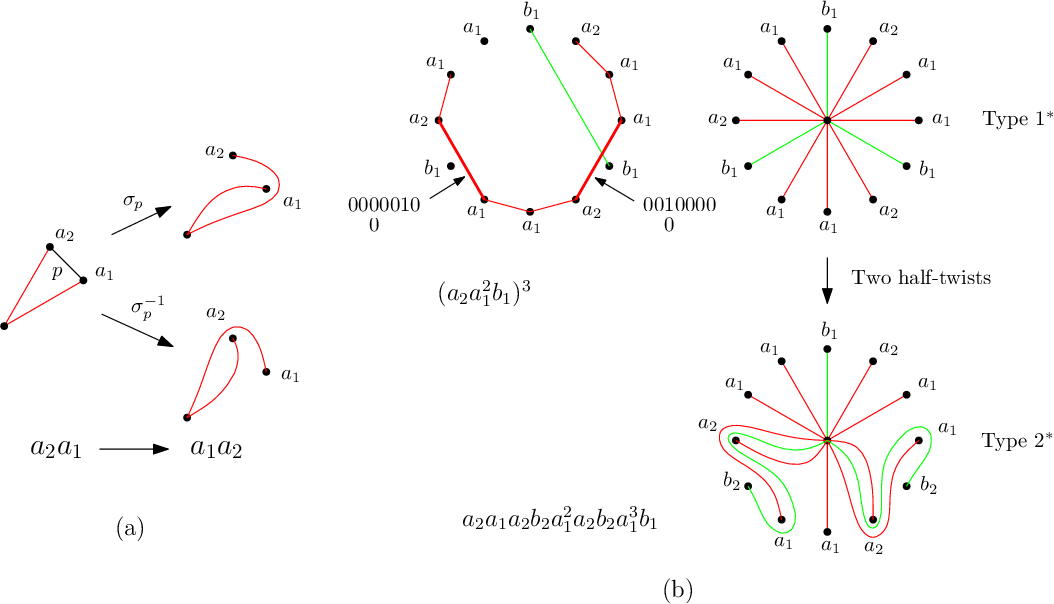}
\caption{Action of half-twists $\s_p$ on the monodromy factorizations: (a) block-variation, and (b) two substitutions.}
\label{fig:Hurwitz-moves}
\end{figure}

\subsection{From monodromy factorizations lifting $(a^3b)$ to those lifting $(ab)^6$.}
Our choice of $(a^3b)^3=1\in\Mod(\T)$ as initial monodromy factorization could be replaced by any other 
monodromy factorization in $\Mod(\T)$.
Correspondingly, our list in Table~\ref{tab:d=1} will be transformed into a list of liftings of the other factorization via
 Hurwitz moves (cf., Theorem \ref{universality}).
Let us give as an example
of such a transformation
from a lifting of $(a^3b)^3$ to a lifting of $(ab)^6$
(other transformations can be performed in a similar way).

For instance, the first factorization, $(a_1a_2^2b_1)^3$, from Table \ref{tab:d=1}, after a cyclic permutation, becomes $(a_2^2b_1a_1)^3$, which
contains three fragments $a_2b_1a_1$. Each of these fragments is transformed by a pair of Hurwitz moves to
$b_1a_1(a_1^{-1}b_1^{-1}a_2b_1a_1)$, where the last Dehn twist $a_1^{-1}b_1^{-1}a_2b_1a_1$ is equal to 
$a_1^{-1}a_2b_1a_2^{-1}a_1$ (since, by the braid relation, 
$b_1^{-1}a_2b_1=a_2b_1a_2^{-1}$).  Thus, the initial factorization is transformed to 
$$a_2b_1a_1(a_1^{-1}a_2b_1a_2^{-1}a_1)a_2b_1a_1(a_1^{-1}a_2b_1a_2^{-1}a_1)a_2b_1a_1(a_1^{-1}a_2b_1a_2^{-1}a_1).
$$
Finally, the conjugation $x\to a_1a_2^{-1}xa_2a_1^{-1}$ applied to all Dehn twists of the latter factorization does not affect the $a_i$, changes $b_1$ to $b_2$
and transforms the factorization to
$a_2b_2a_1b_1a_2b_2a_1b_1a_2b_2a_1b_1$,
which is a lifting of $(ab)^6$.




\end{document}